\documentclass[10pt]{article}
\usepackage{amsmath,amssymb,amsthm,amscd}
\numberwithin{equation}{section}

\def\cE{{\mathcal E}}

\def\cK{{\mathcal K}}

\def\bF{{\mathbf F}}

\def\bF{{\mathbf F}}
\def\bJ{{\mathbf J}}

\def\SL{{\mathbf S}{\mathbf L}}

\def\ZZ{{\mathbb Z}}

\def\RR{{\mathbb R}}
\def\CC{{\mathbb C}}

\newtheorem{prop}{Proposition}[section]
\newtheorem{theo}[prop]{Theorem}
\newtheorem{lemm}[prop]{Lemma}
\newtheorem{coro}[prop]{Corollary}
\newtheorem{rema}[prop]{Remark}

\newtheorem{defi}[prop]{Definition}

\def\begeq{\begin{equation}}
\def\endeq{\end{equation}}

\def\lab{\ }
\def\lab{\label}

\title{K-stability and K\"ahler-Einstein metrics}
\author{Gang Tian\thanks{Supported partially by
a NSF grant}
\\Beijing University and Princeton University}
\date{}

\begin{document}

\maketitle

\tableofcontents

\section{Introduction}

In this paper, we solve a folklore conjecture \footnote{It is often referred as the Yau-Tian-Donaldson conjecture} on Fano manifolds without non-trivial holomorphic vector fields. The main technical ingredient is a conic version of Cheeger-Colding-Tian's theory on compactness of K\"ahler-Einstein manifolds.
This enables us to prove an extension of the partial $C^0$-estimate for K\"ahler-Einstein metrics established in \cite{donaldsonsun12} and \cite{tian12}.

A Fano manifold is a projective manifold with positive first Chern class $c_1(M)$. Its holomorphic fields form a Lie algebra $\eta(M)$.
The folklore conjecture states: {\it If $\eta(M)\,=\,\{0\}$, then $M$ admits a K\"ahler-Einstein metric if and only if
$M$ is K-stable with respect to the anti-canonical bundle $K_M^{-1}$.} Its necessary part was established in \cite{tian97}.
The following gives the sufficient part of this conjecture.

\begin{theo}
\lab{th:main-2}
Let $M$ be a Fano manifold canonically polarized by the anti-canonical bundle $K_{M}^{-1}$. If $M$ is K-stable, then it admits a K\"ahler-Einstein metric.
\end{theo}

An older approach for proving this theorem is to solve the following complex Monge-Ampere equations by the continuity method:
\begin{equation}
\lab{eq:cont-t}
(\omega + \sqrt{-1}\,\partial\bar\partial \varphi )^n \,=\, e^{h - t \varphi} \omega^n,~~~\omega + \sqrt{-1}\,\partial\bar\partial \varphi \,>\, 0,
\end{equation}
where $\omega$ is a given K\"ahler metric with its K\"ahler class $[\omega] = 2\pi c_1(M)$ and $h$ is uniquely determined by
$$ {\rm Ric}(\omega) - \omega \,=\, \sqrt{-1}\,\partial\bar\partial h,~~~ \int_M (e^h - 1 )\, \omega^n\,=\,0.$$
Let $I$ be the set of $t$ for which \eqref{eq:cont-t} is solvable. Then we have known: (1) By the well-known Calabi-Yau theorem, $I$ is non-empty;
(2) In 1983, Aubin proved that $I$ is open \cite{aubin83}; (3) If we can have
an a priori $C^0$-estimate for the solutions of \eqref{eq:cont-t}, then $I$ is closed and consequently, there is a K\"ahler-Einstein metric on $M$.

However, the $C^0$-estimate does not hold in general since there are many Fano manifolds which do not admit any K\"ahler-Einstein metrics.
The existence of K\"ahler-Einstein metrics required certain geometric stability on the underlying Fano manifolds.
In early 90's, I proposed a program towards establishing the existence of K\"ahler-Einstein metrics.
The key technical ingredient of this program is a conjectured partial $C^0$-estimate. If we can affirm this conjecture for the solutions of \eqref{eq:cont-t}, then
we can use the K-stability to derive the a prior $C^0$-estimate and the K\"ahler-Einstein metric. The K-stability was first introduced in \cite{tian97} as
a test for the properness of the K-energy restricted to a finite dimensional family of K\"ahler metrics induced by a fixed embedding by pluri-anti-canonical sections.\footnote{The K-stability was reformulated in more algebraic ways (see \cite{donaldson02}, \cite{paul08} et al.).} However, such a conjecture on partial $C^0$-estimates is still open except for K\"ahler-Einstein metrics.

In \cite{donaldson10}, Donaldson suggested a new continuity method by using conic K\"ahler-Einstein metrics. Those are metrics with conic angle along a divisor.
For simplicity, here we consider only the case of smooth divisors.

Let $M$ be a compact K\"ahler manifold and $D\subset M$ be a smooth divisor.
A conic K\"ahler metric on $M$
with angle $2\pi \beta$ ($0 < \beta \le 1$) along $D$ is a K\"ahler metric on $M\backslash D$ that is asymptotically equivalent along $D$ to
the model conic metric
$$ \omega_{0,\beta} \,=\, \sqrt{-1} \left (\frac{dz_1\wedge d\bar z_1}{|z_1|^{2-2\beta} } \,+\,\sum_{j=2}^n dz_j \wedge d\bar z_j\right ),$$
where $z_1, z_2, \cdots, z_n$ are holomorphic coordinates such that $ D = \{z_1 = 0\}$ locally.
Each conic K\"ahler metric can be given by its K\"ahler form $\omega$ which  represents a cohomology class in $H^{1,1}(M,\CC)\cap H^2(M,\RR)$, referred as
the K\"ahler class $[\omega]$. A conic K\"ahler-Einstein metric is a conic K\"ahler metric which is also
Einstein outside conic points.

In this paper, we only need to consider the following conic K\"ahler-Einstein metrics: Let $M$ be a Fano manifold and $D$ be a smooth
divisor which represents the Poincare dual of $\lambda c_1(M)$. We call $\omega$ a conic K\"ahler-Einstein with conic angle $2\pi \beta$ along $D$
if it has $2\pi c_1(M)$ as its K\"ahler class and satisfies
\begin{equation}
\label{eq:conic-1}
{\rm Ric}(\omega)\,=\,\mu \omega\,+\, 2\pi (1-\beta) [D].
\end{equation}
Here the equation on $M$ is in the sense of currents, while it is classical outside $D$. We will require $\mu > 0$ which is equivalent to
$(1-\beta )\lambda < 1$. As in the smooth case, each conic K\"ahler metric $\omega $ with $[\omega] = 2\pi c_1(M)$ is the curvature of a
Hermitian metric $||\cdot ||$ on the anti-canonical bundle $K_M^{-1}$. The difference is that the Hermitian metric is not smooth, but it
is H\"older continuous.

Donaldson's continuity method was originally proposed as follows: Assume that $\lambda =1$, i.e., $D$ be a smooth anti-canonical divisor. It follows from
\cite{tianyau} that there is a complete Calabi-Yau metric on $M\backslash D$. It was conjectured that this complete metric is the limit of
K\"ahler-Einstein metrics with conic angle $2\pi \beta \mapsto 0$. If this is true, then the set $E$ of $\beta \in (0, 1]$ such that there is a conic K\"ahler
metric satisfying \eqref{eq:conic-1} is non-empty. It is proved in \cite{donaldson10} that $E$ is open. Then we are led to proving that $E$ is closed.

A problem with this original approach of Donaldson arose because we do not know if a Fano manifold $M$ always has
a smooth anti-canonical divisor $D$. Possibly, there are Fano manifolds which do not admit smooth anti-canonical divisors. At least, it seems to be
a highly non-trivial problem whether or not any Fano manifold admits a smooth
anti-canonical divisor. Fortunately, Li and Sun bypassed this problem. Inspired by \cite{jeffresmazzeorubinstein}, they modified Donaldson's original
approach by allowing $\lambda > 1$. They observed that the main existence theorem in \cite{jeffresmazzeorubinstein},
coupled with an estimate on log-$\alpha$ invariants in \cite{berman}, implies the existence of conic K\"ahler-Einstein metrics with conic angle $2\pi \beta$
so long as $\mu = 1-(1-\beta)\lambda$ is sufficiently small.
Now we define $E$ to be set of $\beta \in (1-\lambda^{-1}, 1]$ such that there is a conic K\"ahler
metric satisfying \eqref{eq:conic-1}. Then $E$ is non-empty. It follows from \cite{donaldson10} that $E$ is open.
The difficult part is to prove that $E$ is closed.

The construction of K\"ahler-Einstein metrics with conic angle $2\pi \beta$ can be reduced to solving complex Monge-Ampere equations:
\begin{equation}
\label{eq:conic-2}
(\omega_\beta \,+\, \sqrt{-1} \partial \bar\partial \varphi )^n \,=\, e^{h_\beta - \mu \varphi} \omega^n_\beta,
\end{equation}
where $\omega_\beta$ is a suitable family of conic K\"ahler metrics with $[\omega_\beta] = 2\pi c_1(M)$ and cone angle $2\pi \beta$ along $D$ and
$h_\beta$ is determined
by
$${\rm Ric}(\omega_\beta)\,=\,\mu \,\omega_\beta \,+\, 2\pi (1-\beta)\, [D] + \sqrt{-1} \,\partial\bar\partial h_\beta~~{\rm and}~~\int_M( e^{h_\beta} - 1)\, \omega^n_{\beta}\,=\,0.$$
As shown in \cite{jeffresmazzeorubinstein}, it is crucial for solving \eqref{eq:conic-2} to establish
an a priori $C^0$-estimate
for its solutions. Such a $C^0$-estimate does not hold in general. Therefore, following my program on the existence of K\"ahler-Einstein metrics
through the Aubin's continuity method, we can first establish a partial $C^0$-estimate and then use the K-stability to conclude the $C^0$-estimate, consequently, the
existence of K\"ahler-Einstein metrics on Fano manifolds which are K-stable.

For any integer $\lambda > 0$ and $\beta > 0$, let $\cE(\lambda,\beta)$ be the set of all triples $(M,D,\omega)$, where $M$ is a Fano manifold,
$D$ is a smooth divisor whose Poincare dual is $\lambda\, c_1(M)$ and $\omega$ is a conic K\"ahler-Einstein metric on $M$ with cone angle $2\pi \beta$ along $D$.
For any $\omega\in \cE(\lambda,\beta)$, choose a $C^1$-Hermitian metric $h$ with $\omega$ as its curvature form and any orthonormal basis $\{S_i\}_{0\le i\le N}$
of each $H^0(M, K_M^{-\ell})$ with respect to the induced inner product by $h$ and $\omega$. Then as did in the smooth case, we can introduce a function
\begin{equation}
\lab{hpartial-1}
\rho_{\omega,\ell} (x) \,=\,\sum_{i=0}^N || S_i||_{ h}^2(x).
\end{equation}

One of main results in this paper is the following.

\begin{theo}
\lab{th:main-1}
For any fixed $\lambda$ and $\beta_0 > 1 - \lambda^{-1}$, there are uniform constants $c_k=c(k, n, \lambda,\beta_0)>0$ for $k\ge 1$ and $\ell_i \to \infty$
such that for any $\beta \ge \beta_0$ and $\omega \in \cE(\lambda,\beta)$, we have for $\ell = \ell_i$,

\begin{equation}
\lab{eq: conicpartial}
\rho_{\omega, \ell} \,\ge\, c_\ell\,> 0.
\end{equation}
\end{theo}

In \cite{tian12}, we conjectured that this theorem holds for more general conic K\"ahler metrics.\footnote{Our method in this paper can be also applied to getting 
the partial $C^0$-estimate in this more general case.}

The most crucial tool in proving Theorem \ref{th:main-1} is an extension of a compactness theorem of Cheeger-Colding-Tian for K\"ahler-Einstein metrics.
One needs extra technical inputs to establish such an extension.

As a consequence of Theorem \ref{th:main-1}, we have

\begin{theo}
\lab{th:main-3}
Let $M$ be a Fano manifold with a smooth pluri-anti-canonical divisor $D$ of $K_M^{-\lambda}$.
Assume that $\omega_i$ be a sequence of conic K\"ahler-Einstein metrics with cone angle $2\pi\beta_i$ along $D$ satisfying:
$${\rm Ric}(\omega_i) \,=\, \mu_i \omega_i \,+\,2\pi (1-\beta_i) [D],~~~~~\mu_i \,=\,1- (1-\beta_i) \lambda.$$
where $\mu_i = 1 - (1-\beta_i) \lambda > 0 $. We further assume that $\lim \mu_i = \mu_\infty > 0$ and $(M,\omega_i)$ converge to a length space
$(M_\infty, d_\infty)$ in the Gromov-Hausdorff topology. Then $M_\infty$ is a smooth K\"ahler manifold outside a closed subset $\bar S$ of codimension at least $4$
and $d_\infty$ is induced by a smooth K\"ahler-Einstein metric outside a divisor $D_\infty\subset M_\infty$. Furthermore,
$(M,\omega_i)$ converge to $(M_\infty,\omega_\infty)$ outside $D_\infty$ in the $C^\infty$-topology and
$D$ converges to $D_\infty$ in the Gromov-Hausdorff topology.
\end{theo}

This theorem is needed to finish the proof of Theorem \ref{th:main-2}.

The organization of this paper is as follows: In the next section, we prove an approximation theorem which states any conic K\"ahler-Einstein metrics
can be approximated by smooth K\"ahler metrics with the same lower bound on Ricci curvature. This theorem was not known before and is of interest by itself.
In section 3, we give an extension of my works with Cheeger-Colding in \cite{cheegercoldingtian} to conic
K\"ahler-Einstein manifolds.\footnote{My work with Cheeger and Colding \cite{cheegercoldingtian} 
is definitely needed in establishing the partial $C^0$-estimate which is crucial in
proving Theorem \ref{th:main-2}. } In section 4, we prove the smooth convergence for conic K\"ahler-Einstein metrics. In the smooth case, it is
based on a result of M. Anderson. However, the arguments do not apply for the conic case. We have
to introduce a new method. In Section 5, we prove Theorem \ref{th:main-1}, i.e., the partial $C^0$-estimate for conic K\"ahler-Einstein metrics. In last
section, we prove Theorem \ref{th:main-2}.

The existence of K\"ahler-Einstein metrics on K-stable Fano manifold was first mentioned in my talk during the conference "Conformal and K\"ahler
Geometry" held at IHP in Paris from September 17 to September 21 of 2012. On October 25 of 2012, in my talk at the Blainefest held at Stony Brook University,
I outlined my proof of Theorem \ref{th:main-2}. I learned that X.X. Chen, S. Donaldson and S. Sun posted a short note on October 30 of 2012 in which they
also announced
a proof of Theorem \ref{th:main-2}.

{\bf Acknowledgement}: First I like to thank my former advisor S. T. Yau who brought me the problem of the existence of K\"ahler-Einstein metrics on Fano manifolds
when I was the first-year graduate student in 80s. I like to thank my friends and collaborators J. Cheeger and T. Colding, their foundational regularity theory on Einstein metrics and my joint work with them on K\"ahler-Einstein metrics have
played a crucial role in proving Theorem \ref{th:main-2}. I also like to thank B. Wang, my former postdoctor and collaborator. My joint work with him on almost Einstein metrics is very important in establishing the main technical result in this paper. I also like to thank Chi Li, J. Song and Z.L. Zhang for many useful discussions
in last few years. I am also grateful to Weiyue Ding with whom I had a joint paper \cite{dingtian92} on generalized Futaki invariants. This paper played a very important role in my introducing the K-stability in \cite{tian97}.

\section{Smoothing conic K\"ahler-Einstein metrics}

In this section, we address the question: {\it Can one approximate a conic K\"ahler-Einstein metrics by smooth K\"ahler metrics with Ricci curvature bounded from below}?
For the sake of this paper, we confine ourselves to the case of positive scalar curvature. Our approach can be adapted to other cases where the scalar curvature is non-positive. In fact, the proof is even simpler.

Let $\omega$ be a conic K\"ahler-Einstein metric on $M$ with cone angle $2\pi \beta$ along $D$, where $D$ is a smooth divisor whose Poincare dual is equal to
$\lambda\, c_1(M)$, in particular, $\omega$ satisfies \eqref{eq:conic-1} for $\mu = 1- (1-\beta) \lambda > 0$. For any smooth K\"ahler metric $\omega_0$ with
$[\omega_0]= 2\pi c_1(M)$, we
can write $\omega = \omega_0 + \sqrt{-1}\, \partial\bar\partial \varphi$ for some smooth function $\varphi$ on $M\backslash D$.
Note that $\varphi$ is H\"older continuous
on $M$. Define $h_0$ by
$${\rm Ric}(\omega_0) \,-\, \omega_0\,=\, \sqrt{-1}\, \partial \bar \partial h_0,~~~\int_M (e^{h_0} - 1 ) \,\omega^n_0 \,=\,0.$$
Note that the first equation above is equivalent to
$${\rm Ric}(\omega_0) \,=\, \mu \,\omega_0\,+\,2\pi (1-\beta) [D]\, +\,  \sqrt{-1}\, \partial \bar \partial (h_0 - (1-\beta)\, \log ||S||^2_0 ),$$
where $S$ is a holomorphic section of $K_{M}^{-\lambda}$ defining $D$ and $||\cdot||_0$ is a Hermitian norm on $K^{-\lambda}_M$ with $\lambda \,\omega_0$ as
its curvature. For convenience, we assume that
$$\sup_M ||S||_0 = 1.$$
If $\omega_\beta$ and $h_\beta$ are those in \eqref{eq:conic-2}, then modulo a constant,
$$h_\beta\,=\, h_0 \,-\, (1-\beta) \log ||S||^2_0 \,-\, \log\left (\frac{\omega_\beta^n}{\omega_0^n}\right)\,-\,\mu\,\psi_\beta,$$
where $\omega_\beta \,=\,\omega_0\,+\, \sqrt{-1}\,\partial\bar\partial \psi_\beta$.

It follows from \eqref{eq:conic-1}
\begin{equation}
\label{eq:2-1}
(\omega_0 + \sqrt{-1}\, \partial\bar\partial \varphi )^n \,=\, e^{h_0 - (1-\beta) \log ||S||^2_0 + a_\beta - \mu \varphi} \,\omega^n_0,
\end{equation}
where $a_\beta$ is chosen according to
$$\int_M \left ( e^{h_0 - (1-\beta) \log ||S||^2_0 + a_\beta}\,-\, 1 \right )\, \omega^n_0 \,=\,0.$$
Clearly, $a_\beta$ is uniformly bounded so long as $\beta\,\ge\, \beta_0 \,>\,0$.

The Lagrangian $\bF_{\omega_0, \mu} (\varphi)$ of \eqref{eq:2-1} is given by
\begin{equation}
\label{eq:func-1}
\bJ_{\omega_0} (\varphi)\,-\, \frac{1}{V} \int_M \varphi \,\omega_0^n \,-\,
\frac{1}{\mu} \,\log \left ( \frac{1}{V} \int_M e^{h_0 - (1-\beta) \log ||S||^2_0 + a_\beta - \mu \varphi }\, \omega^n_0\right),
\end{equation}
where $V = \int_M \omega_0^n$ and
\begin{equation}
\label{eq:func-2}
\bJ_{\omega_0} (\varphi)\,=\, \frac{1}{V}\, \sum_{i=0}^{n-1} {i+1\over n+1} \int_M
\sqrt{-1} \, \partial \varphi\wedge \overline{\partial} \varphi \wedge
\omega^i_0\wedge \omega_\varphi^{n-i-1},
\end{equation}
where $\omega_\varphi= \omega_0 + \sqrt{-1} \,\partial\bar\partial \varphi$. Note that $\bF_{\omega_0,\mu}$ is well-defined for any continuous
function $\varphi$.

Let us recall the following result
\begin{theo}
\label{th:bo}
If $\omega = \omega_\varphi$ is a conic K\"ahler-Einstein with conic angle $2\pi \beta$ along $D$, then $\varphi$ attains the minimum of the functional
$\bF_{\omega_0,\mu}$ on the space $\cK_{\beta} (M,\omega_0)$ which consists of all smooth functions $\psi$ on $M\backslash D$ such that $\omega_\psi$ is
a conic K\"ahler metric with angle $2\pi \beta$ along $D$. In particular, $\bF_{\omega_0,\mu}$ is bounded from below.
\end{theo}
One can find its proof in \cite{berndtsson}. An alternative proof may be given by extending the arguments in \cite{dingtian91} to conic K\"ahler metrics.

\begin{coro}
\label{coro:bo}
If $\mu < 1$, then there are $\epsilon > 0$ and $C_\epsilon > 0$, which may depend on $\omega$ and $\mu$, such that for any $\psi \in \cK_\beta(M,\omega_0)$,
we have for any $t\in (0,\mu]$\footnote{The corresponding $\beta_t$ is defined by $(1-t) = (1-\beta_t)\lambda$.}
\begin{equation}
\label{eq:func-3}
\bF_{\omega_0, t} (\psi )\,\ge \, \epsilon \,\bJ_{\omega_0} (\psi)\,-\, C_\epsilon.
\end{equation}
\end{coro}
\begin{proof}
It follows from the arguments of using the log-$\alpha$-invariant in \cite{lisun} that $\bF_{\omega_0,t}$ is proper for $t> 0$ sufficiently small.
Let $\omega=\omega_\varphi$ be the conic K\"ahler-Einstein metric with angle $\beta$ along $D$. Then $\varphi$ satisfies \eqref{eq:2-1}. Since $M$ does not
admit non-zero holomorphic fields,\footnote{Even if $M$ does have non-trivial holomorphic fields, there should be no holomorphic fields which are tangent to $D$.
This is sufficient for rest of the proof.}it follows from \cite{donaldson10} that \eqref{eq:2-1} has a solution $\bar \varphi$ when $\mu$ is replaced by
$\bar\mu\,=\,\mu+\delta$ for $\delta >$ sufficiently small. Hence, by Theorem \ref{th:bo},
$\bF_{\omega_0,\bar{\mu}} $ is bounded from below. Then this corollary follows from Proposition 1.1 in
\cite{lisun}\footnote{In \cite{lisun}, the reference metric $\omega_0$ is a conic K\"ahler metric while ours is a smooth metric, however, the arguments apply with
slight modification.}

\end{proof}

Now we consider the following equation:
\begin{equation}
\label{eq:2-2}
(\omega_0 + \sqrt{-1}\, \partial\bar\partial \varphi )^n \,=\, e^{h_\delta - \mu \varphi}\, \omega^n_0,
\end{equation}
where
$$h_\delta \,=\, h_0 - (1-\beta) \log (\delta + ||S||^2_0) + c_\delta$$
for some constant $c_\delta$ determined by
$$\int_M \left( e^{ h_0 - (1-\beta) \log (\delta + ||S||^2_0) + c_\delta} - 1 \right ) \,\omega_0^n \,=\,0.$$
Clearly, $c_\delta$ is uniformly bounded. If $\varphi_\delta$ is a solution, then we get a smooth K\"ahler metric
$$\omega_\delta\, =\, \omega_0\, +\, \sqrt{-1}\, \partial\bar\partial \varphi_\delta.$$
Its Ricci curvature is given by
$${\rm Ric} (\omega_\delta) \,=\, \mu \,\omega_\delta \,+\, \frac{\delta (1-\beta) \lambda }{\delta + ||S||^2_0}
\,\omega_0\,+\, \delta(1-\beta)\,\frac{DS\wedge \overline{DS}}{(\delta + ||S||^2_0)^2},
$$
where $DS$ denotes the covariant derivative of $S$ with respect to the Hermitian metric $||\cdot||_0$. In particular, the Ricci curvature of $\omega_\delta$
is greater than $\mu$ whenever $\beta < 1$ and $\delta > 0$.\footnote{This observation is crucial in our approximating the conic K\"ahler-Einstein
metric $\omega$ and first appeared in the slides of my talk at SBU on October 25, 2012. The arguments in establishing the existence of $\omega_\delta$
is known for long and identical to the one I used in \cite{tian97}.}

We will solve \eqref{eq:2-2} for such $\omega_\delta$'s and show that they converge to the conic
K\"ahler-Einstein metric $\omega$ in a suitable sense.

To solve \eqref{eq:2-2}, we use the standard continuity method:
\begin{equation}
\label{eq:2-3}
(\omega_0 + \sqrt{-1}\, \partial\bar\partial \varphi )^n \,=\, e^{h_{\delta} - t \varphi}\, \omega^n_0.
\end{equation}
Define $I_\delta$ to be the set of $t\in [0,\mu]$ for which \eqref{eq:2-3} is solvable. By the Calabi-Yau theorem, $0\in I_\delta$.

We may assume $\mu < 1$, otherwise, we have nothing more to do.

\begin{lemm}
\label{lemm:2-1}
The interval $I_\delta$ is open.
\end{lemm}
\begin{proof}
If $t\in I_\delta$ and $\varphi$ is a corresponding solution of \eqref{eq:2-3}, then the Ricci curvature of the associated metric $\omega_\varphi$ is equal to
$$ t \,\omega_\varphi \,+\, \left ( (\mu-t) + \frac{\delta (1-\beta) \lambda }{\delta + ||S||^2_0}
\right)\,\omega_0\,+\, \delta(1-\beta)\,\frac{DS\wedge \overline{DS}}{(\delta + ||S||^2_0)^2}.$$
So ${\rm Ric}(\omega_\varphi)  \,> \, t \,\omega_\varphi$. By the well-known Bochner identity, the first non-zero eigenvalue of $\omega_\varphi$ is strictly bigger than $t$. It implies that the linearization $\Delta_t + t $ of \eqref{eq:2-3} at $\varphi$ is invertible,
where $\Delta_t$ is the Laplacian of $\omega_\varphi$.
By the Implicit Function Theorem, \eqref{eq:2-3} is solvable for any $t'$ close to $t$, so $I_\delta$ is open.
\end{proof}

Therefore, we only need to prove that $I_\delta$ is closed. This is amount to a priori estimates for any derivatives of the solutions of \eqref{eq:2-3}.
As usual, by using known techniques in deriving higher order estimates, we need to bound only $J_{\omega_0}(\varphi)$ for any solution $\varphi$ of \eqref{eq:2-3}
(cf. \cite{tian97}, \cite{tian98}). The following arguments are identical to those for proving that the properness of ${\bf F}_{\omega_0,1}$ implies the existence
of the K\"ahler-Einstein metrics in Theorem 1.6 of \cite{tian97}.

We introduce
\begin{equation}
\label{eq:func-4}
\bF_{\delta, t} (\varphi)\,=\, \bJ_{\omega_0} (\varphi)\,-\, \frac{1}{V} \,\int_M \varphi \,\omega_0^n \,-\,
\frac{1}{t} \,\log \left ( \frac{1}{V} \int_M e^{h_\delta  - t \varphi } \,\omega^n_0\right).
\end{equation}
This is the Lagrangian of \eqref{eq:2-3}.

\begin{lemm}
\label{lemm:2-2}
There is a constant $C$ independent of $t$ satisfying: For any smooth family of $\varphi_s$ ($s \in [0,t]$) such that $\varphi=\varphi_t$ and
$\varphi_s$ solves \eqref{eq:2-3} with parameter $s$, we have
$$\bF_{\delta, t} (\varphi) \,\le\, C.$$
\end{lemm}
\begin{proof}
First we observe
\begin{equation}
\label{eq:derivative}
\bF_{\delta, s} (\varphi_s)\,=\, \bJ_{\omega_0} (\varphi_s)\,-\, \frac{1}{V} \int_M \varphi_s \,\omega_0^n .
\end{equation}
So its derivative on $s$ is given by
$$\frac{d}{ds}\, \bF_{\delta, s} (\varphi_s) \,=\,\frac{1}{s V} \,\int_M \varphi_s \,(\omega_0 + \sqrt{-1}\, \partial\bar\partial \varphi_s)^n.$$
Here we have used the fact
$$\int_M (\dot{\varphi_s}\,+\,s\, \varphi_s) \,(\omega_0 + \sqrt{-1}\, \partial\bar\partial \varphi_s)^n\,=\,0$$
This follows from differentiating \eqref{eq:2-3} on $s$.

We will show that the derivative in \eqref{eq:derivative} is bounded from above. Without loss of the generality, we may assume that $s \ge s_0 > 0$. Then we have
$${\rm Ric}(\omega_{\varphi_s} )\,\ge \, s\,\omega_{\varphi_s} \,\ge\, s_0\,\omega_{\varphi_s},$$
and consequently, the Sobolev constant of $\omega_{\varphi_s}$ is uniformly bounded. By the standard Moser iteration, we have (cf. \cite{tian98})
$$ - \inf_M \varphi_s \,\le \, -\, \frac{1}{ V} \,\int_M \varphi_s \,(\omega_0 + \sqrt{-1} \partial\bar\partial \varphi_s)^n\,+\,C'.$$
Since $\inf_M \varphi_s \,\le\, 0$, we get
$$\frac{d }{ds}\,\bF_{\delta, s} (\varphi_s) \,\le \, s_0^{-1} C'.$$
The lemma follows from integration along $s$.
\end{proof}

Next we observe for any $t\le \mu$
$$h_\delta \,= \, h_0 - (1-\beta) \log (\delta + ||S||^2_0) + c_\delta \, \le\, h_0 - (1-\beta_t) \log ||S||^2_0 + c_\delta.$$
Hence, by Corollary \ref{coro:bo}, we have
$$\bF_{\delta, t} (\psi) \,\ge\,  \epsilon \,\bJ_{\omega_0} (\psi)\,-\, C_\epsilon\,-\, \frac{ c_\delta - a_\beta }{t}.$$
Since both $c_\delta $ and $a_\beta $ are uniformly bounded, combined with Lemma \ref{lemm:2-2}, we conclude that
$J_{\omega_0}(\varphi)$ is uniformly bounded for any solution $\varphi$ of \eqref{eq:2-3}.\footnote{Here we also used the fact that $J_{\omega_0}(\varphi)$ is
automatically bounded for $t> 0$ sufficiently small.} Thus we have proved

\begin{theo}
\label{th:2-1}
For any $\delta > 0$, \eqref{eq:2-2} has a unique smooth solution $\varphi_\delta$. Consequently, we have a K\"ahler metric
$\omega_\delta=\omega_0+\sqrt{-1}\,\partial\bar\partial \varphi_\delta$ with Ricci curvature greater than or equal to $\mu$.
\end{theo}

Next we examine the limit of $\omega_\delta$ or $\varphi_\delta$ as $\delta$ tends to $0$.
First we note that for the conic K\"ahler-Einstein metric $\omega$ with cone angle $2\pi \beta$ along $D$ given above, there is a uniform constant $c=c(\omega)$ such that
$\sup_M |\varphi_\delta|\,\le\, c$. Using ${\rm Ric}(\omega_\delta)\, \ge\, \omega_\delta $ and the standard computations, we have
$$\Delta \log {\rm tr}_{\omega_\delta}(\omega_0) \,\ge \, - a \, {\rm tr}_{\omega_\delta}(\omega_0),$$
where $\Delta$ is the Laplacian of $\omega_\delta$ and $a$ is an upper bound of the bisectional curvature of $\omega_0$. If we put
$$u \,=\, {\rm tr}_{\omega_\delta}(\omega_0) - (a+1) \,\varphi_\delta,$$
then it follows from the above
$$ \Delta u \,\ge \, u - n - (a+1)\, c .$$
Hence, we have
$$ u \,\le\, n + (a+1)\, c,$$
this implies
$$ C^{-1} \,\omega_0\,\le \, \omega_\delta ,$$
where $C \,=\,n + 2 (a+1)\, c$. Using the equation \eqref{eq:2-3}, we have
\begin{equation}
\label{eq:2-side-esti}
C^{-1}\,\omega_0 \,\le\,\omega_\delta\,\le \, C'\, (\delta + ||S||^2)^{-(1-\beta)} \,\omega_0,
\end{equation}
where $C'$ is a constant depending only on $a$ and $\omega_0$. Since $\beta > 0$, the above estimate on $\omega_\delta = \omega_0 +
\sqrt{-1}\, \partial\bar\partial \varphi_\delta$ gives the uniform H\"older continuity of $\varphi_\delta$. Furthermore, using the Calabi estimate for the 3rd derivatives and the standard regularity theory, we can prove (cf. \cite{tian98}): For any $ l > 2$ and a compact subset $K \subset M\backslash D$, there is a uniform constant $C_{l, K}$
such that
\begin{equation}
\label{eq:h-d}
||\varphi_\delta||_{C^l(K)}\, \le \, C_{l, K}.
\end{equation}
Then we can deduce from the above estimates:

\begin{theo}
\label{th:GH}
The smooth K\"ahler metrics $\omega_\delta$ converge to $\omega$ in the Gromov-Hausdorff topology on $M$ and in the smooth topology outside $D$.
\end{theo}

\begin{proof}
It suffices to prove the first statement: $\omega_\delta$ converge to $\omega$ in the Gromov-Hausdorff topology.
Since $\omega_\delta$ has Ricci curvature bounded from below by a fixed $\mu > 0$, by the Gromov Compactness Theorem, any sequence of $(M,\omega_\delta)$
has a subsequence converging to a length space $(\bar M,\bar d)$ in the Gromov-Hausdorff topology. We only need to prove that any such a limit
$(\bar M,\bar d)$ coincides with $(M,\omega)$. Without loss of generality, we may assume that $(M,\omega_\delta)$
converge to $(\bar M,\bar d)$ in the Gromov-Hausdorff topology. By the estimates on derivatives in \eqref{eq:h-d},
$\bar M$ contains an open subset $U$ which can be identified with $M\backslash D$, moreover, this identification $\iota: M\backslash D\mapsto U$
is an isometry between $(M\backslash D, \omega|_{M\backslash D})$ and $(U,\bar d|_U)$.
On the other hand, since $\omega$ is a conic metric with angle $2\pi \beta \le 2\pi$ along $D$, one can easily show by standard arguments that
$M\backslash D$ is geodesically convex with respect to $\omega$. Then it follows from \eqref{eq:2-side-esti} that $M$ is the metric completion of $M\backslash D$
and $\iota$ extends to a Lipschtz map from $(M,\omega)$ onto $(\bar M,\bar d)$, still denoted by $\iota$. In fact, the Lipschtz constant is $1$.

We claim that $\iota$ is an isometry. This is equivalent to the following: For any $p$ and $q$ in $M\backslash D$,
$$d_\omega(p,q)\,=\,\bar d(\iota(p), \iota(q)).$$
It also follows from \eqref{eq:2-side-esti} that $\bar D = \iota(D)$ has Hausdorff measure $0$ and is the Gromov-Hausdorff limit of $D$ under the convergence of
$(M,\omega_\delta)$ to $(\bar M,\bar d)$.
To prove the above claim, we only need to prove: For any $\bar p, \bar q \in \bar M\backslash\bar D$, there is a minimizing geodesic
$\gamma \subset  \bar M\backslash\bar D$ joining $\bar p$ to $\bar q$. Its proof is based on a relative volume comparison estimate due to Gromov
(\cite{gromov}, p 523, (B)).
\footnote{I am indebted to Jian Song for this reference. He seems to be the first of applying such an estimate to studying the convergence problem in
K\"ahler geometry.}
We will prove it by contradiction. If no such a geodesic joins $\bar p$ to $\bar q$, then
$$\bar d(\bar p,\bar q)\,< \, d_\omega(p,q),$$
where $\bar p = \iota(p)$ and $\bar q = \iota (q)$. Then there is a $r>0$ satisfying:

\vskip 0.1in
\noindent
(1) $B_r(\bar p, \bar d)\cap \bar D=\emptyset$
and $B_r(\bar q, \bar d)\cap \bar D=\emptyset$, where $B_r(\cdot, \bar d)$ denotes a geodesic ball in $(\bar M, \bar d)$;

\vskip 0.1in
\noindent
(2) $\bar d(\bar x,\bar y) < d_\omega (x,y)$, where $\bar x = \iota(x)\in B_r(\bar p, \bar d) $ and $\bar y = \iota (y)\in B_r(\bar q, \bar d)$.

\vskip 0.1in
It follows from (1) and (2) that any minimizing geodesic $\gamma$ from $\bar x$ to $\bar y$ intersects with $\bar D$.
By choosing $r$ sufficiently small, we may have
$$B_r(\bar p, \bar d) = \iota (B_r(p, \omega))~~~{\rm and}~~~B_r(\bar q, \bar d)=\iota(B_r(q, \omega)).$$
Choose a small tubular neighborhood $T$ of $D$ in $M$ whose closure is disjoint from both $B_r(p, \omega)$ and $B_r(q, \omega)$.
It is easy to see that $T$ can be chosen to have the volume of $\partial T$
as small as we want. Now we choose $p_\delta, q_\delta\in M$ and neighborhood $T_\delta$ of $D$ with respect to $\omega_\delta$ such that
in the Gromov-Haudorff convergence,
$$\lim_{\delta\to 0+} p_\delta \,= \,\bar p\,,~~~\lim_{\delta\to 0+} q_\delta\, =\, \bar q\,,~~~\lim_{\delta\to 0+} T_\delta\,=\,\iota(T)\,.$$
It follows
$$\lim_{\delta\to 0+} Vol(\partial T_\delta,\omega_\delta)\,=\, Vol(\partial T, \omega).$$
Also, for $\delta$ sufficiently small,
$B_r(p_\delta, \omega_\delta)$, $B_r(q_\delta, \omega_\delta)$ and $T_\delta$ are mutually disjoint. Clearly, any minimizing geodesic
$\gamma_\delta$ from any $w\in B_r(p_\delta, \omega_\delta)$ to $z\in B_r(q_\delta, \omega_\delta)$ intersects with $T_\delta$,
so by Gromov's estimate (\cite{gromov}, p523, (B)),
$$c\, r^{2n} \,\le\, Vol(B_r(q_\delta, \omega_\delta),\omega_\delta)\,\le\, C\,Vol(\partial T_\delta,\omega_\delta),$$
where $c$ depends only on $\beta$ and $C$ depends only on $\mu$, $n$, $r$. This leads to a contradiction because
$Vol(\partial T_\delta,\omega_\delta)$ converge to $Vol(\partial T,\omega)$ which can be made as small as we want.
Thus, $\iota$ is an isometry and our theorem is proved.

\end{proof}
Finally, we exam the limit of $\rho_{\omega_\delta, \ell}$ for any $\ell >0$.

First we associate a Hermitian norm $||\cdot||^2_0$ to $\omega_0$:
For any section $\sigma$ of $K_M^{-1}$, in local coordinates $z_1,\cdots, z_n$, we can write
$$\sigma \,=\, f \frac{\partial}{\partial z_1}\wedge \cdots \wedge \frac{\partial}{\partial z_n},$$
then
$$ ||\sigma||^2_0\,=\, e^{h_0} \, \det(g_{i\bar j}) \,|f|^2,$$
where $\omega_0\,=\,\sqrt{-1}\, g_{i\bar j} \,dz_i\wedge d\bar z_j$.
The curvature form of $||\cdot||_0^2$ is $\omega_0$.

Similarly, we can associate a Hermitian norm $||\cdot||^2_\delta$ whose curvature is $\omega_\delta$.
Using \eqref{eq:2-2}, we see
$$||\cdot  ||_\delta^2\,=\, e^{c_\delta' - \varphi_\delta } \,||\cdot||_0^2,$$
where $c_\delta'$ satisfies
$$\int_M \left ( e^{h_0 -  \varphi_\delta + c_\delta' } \,-\, 1\right ) \, \omega_0^n \,=\,0.$$
Then as $\delta \to 0$, Hermitian norms $||\cdot||^2_\delta$ converge to the Hermitian norm on $K^{-1}_M$:\footnote{For simplicity of notations, we do not make
explicit the dependence of $||\cdot||_\delta$ and $||\cdot||$ on $\mu$.}
$$||\cdot||^2 \,=\, e^{c' - \varphi}\, ||\cdot||_0^2,$$
where $\varphi $ is the solution of \eqref{eq:2-1} and $c'$ satisfies:
$$\int_M \left ( e^{h_0 - \varphi + c'} \,-\,1\right )\,\omega_0^n\,=\,0.$$
If we denote by $||\cdot||_\beta^2$ the Hermitian norm on $K_M^{-1}$ defined by the volume form of $\omega_\varphi$, then
$$||\cdot||^2\,=\, e^{c' - a_\beta}\, ||S||_\beta^{2(1-\beta)} \, ||\cdot||^2_\beta.$$

\begin{theo}
\label{th:2-5}
For any $\ell >0$, let $<\cdot,\cdot>_\delta$ be the inner product on $H^0(M,K_M^{-\ell})$ induced by $\omega_\delta$ and the Hermitian metric $||\cdot||_\delta^2 $
on $K_M^{-1}$. Then as $\delta $ tends to $0$, $<\cdot,\cdot>_\delta$ converge to the corresponding inner product by the Hermitian metric
$||\cdot||^2$ and $\omega$. In particular, when $\ell $ is sufficiently large,
$\rho_{\omega_\delta, \ell}$ converge to $\rho_{\omega,\ell}$.
\end{theo}
\begin{proof}
We have seen above that $\varphi_\delta$ converges to $\varphi$ in a H\"older continuous norm. It follows that the volume forms $\omega_\delta^n$ converge to
$\omega^n$ in the $L^p$-topology for any given $p \in (1, (1-\beta)^{-1})$ and the Hermitian metrics $||\cdot||_\delta^2$ converge to $||\cdot||^2$.
Since the inner products $<\cdot,\cdot>_\delta$ are defined by these Hermitian metrics and volumes forms, the theorem follows easily.
\end{proof}

\section{An extension of Cheeger-Colding-Tian}

In this section, we show a compactness theorem on conic K\"ahler-Einstein metrics. This theorem, coupled with the smooth convergence result in the next section,
extends a result of Cheeger-Colding-Tian \cite{cheegercoldingtian} on smooth K\"ahler-Einstein metrics. In fact, our proof makes use of results in \cite{cheegercoldingtian} with injection of some new technical ingredients.

Let $\omega_i$ be a sequence of conic K\"ahler-Einstein metrics with cone angle $2\pi\beta_i$ along $D$, so we have
$${\rm Ric}(\omega_i) \,=\, \mu_i \omega_i \,+\,2\pi (1-\beta_i) [D],~~~~~\mu_i \,=\,1- (1-\beta_i) \lambda.$$
We assume that $\beta_i \ge \epsilon > 0$ and $\lim \beta_i = \beta_\infty$, it follows $\lim \mu_i = \mu_\infty >0$.

For each $\omega_i$, we use Theorem \ref{th:GH} to get a smooth K\"ahler metric $\tilde \omega_i$ satisfying:

\vskip 0.1in
\noindent
{\bf A1}. Its K\"ahler class
$[\tilde\omega_i]\, =\, 2 \pi c_1(M)$;

\vskip 0.1in
\noindent
{\bf A2}. Its Ricci curvature ${\rm Ric}(\tilde\omega_i) \,\ge \, \mu_i$;

\vskip 0.1in
\noindent
{\bf A3}. The Gromov-Hausdorff distance $d_{GH}(\omega_i,\tilde \omega_i)$ is less that $1/i$.
\vskip 0.1in

By the Gromov compactness theorem, a subsequence of $(M, \tilde\omega_i)$ converges to a metric space $(M_\infty,d_\infty)$ in the Gromov-Hausdorff topology.
For simplicity, we may assume that $(M, \tilde\omega_i)$ converges to $(M_\infty,d_\infty)$. It follows from (3) above that $(M,\omega_i)$
also converges to $(M_\infty,d_\infty)$ in the Gromov-Hausdorff topology.

\begin{theo}
\label{th:3-1}
There is a closed subset ${\cal S}\subset M_\infty$ of Hausdorff codimension at least $2$ such that $M_\infty\backslash {\cal S}$ is a smooth K\"ahler manifold and $d_\infty$ is
induced by a K\"ahler-Einstein metric $\omega_\infty$ outside ${\cal S}$, that is,
$${\rm Ric}(\omega_\infty)\,=\,\mu_\infty\, \omega_\infty ~~~~{\rm on}~~M_\infty\backslash {\cal S}.$$
If $\beta_\infty < 1$, then $\omega_i$ converges to $\omega_\infty$ in the $C^\infty$-topology outside ${\cal S}$.
Moreover, if $\beta_\infty=1$, the set ${\cal S}$ is of codimension at least $4$
and $\omega_\infty$ extends to a smooth K\"ahler metric on $M_\infty\backslash {\cal S}$.
\end{theo}

This theorem is essentially due to Z.L. Zhang and myself \cite{tianzhang}. In this joint work, we develop a regularity theory for conic Einstein metrics
which generalizes the work of Cheeger-Colding and Cheeger-Colding-Tian. Here, for completion and convenience, we give an alternative proof by
using the approximations from last section.

\begin{proof}
Using the fact that $(M_\infty, d_\infty)$ is the Gromov-Hausdorff limit of $(M,\tilde\omega_i)$, we can deduce from \cite{cheegercolding}
the existence of tangent cones at every $x\in M_\infty$. More precisely, given any $x\in M_\infty$, for any $r_i\mapsto 0$,
by taking a subsequence if necessary, $(M_\infty, r_i^{-2} d_\infty, x)$ converges to a tangent cone ${\cal C}_x$ at $x$.
Define ${\cal R}$ to be the set of
all points $x$ in $M_\infty$ such that some tangent cone ${\cal C}_x$ is isometric to $\RR^{2n}$.

First we prove that ${\cal R}$ is open. If $\beta_\infty = 1$, then $\lim \mu_i =1$. Since
$$[\tilde\omega_i] \,=\, 2\pi\, c_1(M)~~~{\rm and}~~~{\rm Ric}(\tilde\omega_i)\,\ge\,\mu_i\,\tilde\omega_i,$$
by the arguments in the proof of Theorem 6.2 in \cite{tianwang},
one can show that $(M,\tilde\omega_i)$ is a sequence of almost K\"ahler-Einstein metrics in the sense of
\cite{tianwang}. Then it follows from Theorem 2 in \cite{tianwang} that $M_\infty$ is smooth outside
a closed subset $ {\cal S}$ of codimension at least $4$ and $d_\infty$ is induced by a smooth K\"ahler-Einstein metric $\omega_\infty$
on $M_\infty\backslash  {\cal S}$.

Now assume that $\beta_\infty < 1$. Note that $(M,\omega_i)$
also converge to $(M_\infty,d_\infty)$ in the Gromov-Hausdorff topology. Let $\{x_i\}$ be a sequence of points in $M$ which
converge to $x \in {\cal R}$ during $(M, \omega_i)$'s converging to $(M_\infty, d_\infty)$. Since $x\in {\cal R}$, there is a tangent cone
${\cal C}_x$ of $(M_\infty, d_\infty)$ at $x$ which is isometric to $\RR^{2n}$. It follows that for any $\epsilon >0$, there is a $r=r(\epsilon)$ such that
$$\frac{Vol(B_r(x,d_\infty))}{r^{2n}} \,\ge \, c(n) - \epsilon,$$
where $c(n)$ denotes the volume of the unit ball in $\RR^{2n}$.
On the other hand, if $y_i\in D $, then by the Bishop-Gromov volume comparison, for any $\tilde r > 0 $, we have
$$\frac{Vol(B_{\tilde r} (y_i, \omega_i))}{\tilde r^{2n} }\,\le\, c(n) \,\beta_i .$$
It also follows from the Boshop-Gromov volume comparison that there is an $N\,=\,N(\epsilon)$
such that for any small $\bar r\in (0,r/N)$ and $y_i\in B_{\bar r} (x_i,\omega_i)$, we have
$$1 - \epsilon\,\le\, \frac{Vol(B_{r} (y_i, \omega_i))} {Vol (B_{r} (x_i, \omega_i))} \,\le\, 1 + \epsilon.$$

Now we claim that if $\bar r = r/ N$, we have $B_{\bar r}(x_i,\omega_i)\cap D = \emptyset$. If this claim is false, say $y_i\in B_{\bar r}(x_i,\omega_i)\cap D $,
then for $i$ sufficiently large, we can deduce from the above and a result of Colding \cite{colding} on the volume convergence in the Gromov-Hausdorff topology
$$ c(n) - 2 \epsilon \,\le\, \frac{Vol(B_r(x_i, \omega_i))}{r^{2n}} \,\le\, (1+\epsilon) \,\frac{Vol(B_r(y_i, \omega_i))}{r^{2n}} \,\le\, c(n) ( 1 + \epsilon) \,\beta_i.$$
Then we get a contradiction if $\epsilon$ is chosen sufficiently small. The claim is proved.

Since $B_{\bar r}(x_i,\omega_i)$ is contained in the smooth part of $(M,\omega_i)$ and its volume
is sufficiently close to that of an Euclidean ball, the curvature of $\omega_i$ is uniformly bounded on the half ball $B_{3\bar r/4}(x_i,\omega_i)$
(cf. \cite{anderson}). It follows that $\omega_i$ restricted to $B_{\bar r/2}(x_i,\omega_i)$ converge to a smooth K\"ahler-Einstein metric on $B_{\bar r/2}(x,d_\infty)$ and
$B_{\bar r/2}(x,d_\infty) \subset {\cal R}$. So ${\cal R}$ is open and $d_\infty$ restricted to ${\cal R}$ is induced by a smooth K\"ahler-Einstein metric
$\omega_\infty$.

The rest of the proof is standard in view of \cite{cheegercoldingtian}.

Let ${\cal S}_k$ ($k=0,1,\cdots,2n-1$) denote the subset of $M_\infty$
consisting of points for which no tangent cone splits off a factor, $\RR^{k+1}$, isometrically. Clearly,
${\cal S}_0 \subset {\cal S}_1\subset \cdots \subset {\cal S}_{2n-1}$. It is proved by Cheeger-Colding that ${\cal S}_{2n-1} = \emptyset$, $\dim {\cal S}_k \le k$ and ${\cal S}\,=\,{\cal S}_{2n-2}$.
Moreover, if $\beta_\infty = 1$, it follows from \cite{tianwang} that ${\cal S}\,=\,{\cal S}_{2n-4}$. Then we have proved this theorem.
\end{proof}

Using the same arguments in \cite{cheegercoldingtian}, one can show:

\begin{theo}
\label{th:tangentcone}
Let ${\cal C}_x$ be a tangent cone of $M_\infty$ at $x\in {\cal S}$, then we have

\vskip 0.1in
\noindent
{\bf C1}. Each ${\cal C}_x$ is regular outside a closed subcone ${\cal S}_x$ of complex codimension at least $1$. Such a ${\cal S}_x$ is the singular set of ${\cal C}_x$;

\vskip 0.1in
\noindent
{\bf C2}. ${\cal C}_x = \CC^k\times {\cal C}_x'$, in particular, ${\cal S}_{2k+ 1}\, =\, {\cal S}_{2k}$. We will denote $o$ the vortex of ${\cal C}_x$;

\vskip 0.1in
\noindent
{\bf C3}. There is a natural K\"ahler Ricci-flat metric $g_x$ whose K\"ahler form $\omega_x$ is $\sqrt{-1}\, \partial\bar\partial \rho_x^2$
on ${\cal C}_x \backslash {\cal S}_x$ which is also a cone metric, where $\rho_x$ denotes the distance function from the vertex of ${\cal C}_x$;

\vskip 0.1in
\noindent
{\bf C4}. For any $x \in {\cal S}_{2n-2}$, ${\cal C}_x \,=\, \CC^{n-1}\times {\cal C}_x'$, where ${\cal C}_x'$ is a 2-dimensional flat cone of angle $2 \pi \bar \mu$
such that $0 < \bar\beta_\infty \le \bar\mu \le \beta_\infty$ and $(1-\bar\mu) = m (1-\beta_\infty)$ for some integer $m \ge 1$, where 
$\bar \beta_\infty$ depends only on $\beta_\infty$.
\end{theo}

In fact, {\bf C1}, {\bf C2} and {\bf C3} follow directly from results in \cite{cheegercoldingtian}. The proof of {\bf C4} uses the slicing argument
in \cite{cheegercoldingtian} for proving that ${\cal S}_{2n-2}\,=\,\emptyset$ in the case of smooth K\"ahler-Einstein metrics. In our new case, the conic
singularity of $\omega_i$ along $D$ may contribute a term close to $2\pi (1-\beta_i)$ in the slicing argument, this is how we can conclude that
${\cal C}_x'$ is a 2-dimensional flat cone of angle $2\pi \bar\mu$. The bounds on $\bar\mu$ follow from the Bishop-Gromov volume comparison. Note that
$\bar\beta_\infty$ depends only on the diameter and volume of $M_\infty$. Hence, there are only finitely many of such $\bar\mu$.

Next we state a corollary of Theorem \ref{th:GH}:
\begin{lemm}
\label{lemm:3-1}
There is a uniform bound on the Sobolev constants of $(M,\omega_i)$, that is, there is a
constant $C$ such that for any $f \in C^1(M,\RR)$,
\begin{equation}
\label{eq:sobolev}
\left (\int_M |f|^{\frac{2n}{n-1}} \right)^{\frac{n-1}{n}} \,\omega_i^n \,\le\, C \int_M (|d f|_{\omega_i}^2\,+\,|f|^2)\,\omega_i^n.
\end{equation}
\end{lemm}
\begin{proof}
By Theorem \ref{th:GH}, for any $i$, there is a sequence of smooth K\"ahler metrics $\omega_{i,\delta}$ converging to $\omega_i$ in the Gromov-Hausdorff topology
and ${\rm Ric}(\omega_{i,\delta}) \,\ge\, \mu_i \,\omega_{i,\delta}$. Since the volume of $\omega_{i,\delta}$ is fixed, it is well-known that \eqref{eq:sobolev} holds uniformly for $\omega_{i,\delta}$. Then the lemma follows by taking $\delta\to 0$.
\end{proof}

\section{Smooth convergence}

We will adopt the notations from last section, e.g., $\omega_i$ is a conic K\"ahler-Einstein metric on $M$ with angle $2\pi \beta_i$ along $D$ as before.
The main result of this section is to show that $\omega_i$ converge to $\omega_\infty$ outside a close subset of codimension at least $2$. This is crucial
for our establishing the partial $C^0$-estimate for conic K\"ahler-Einstein metrics as well as finishing the proof Theorem \ref{th:main-2}.
This is related to the limit of $D$ when $(M,\omega_i)$ converges to $(M_\infty, d_\infty)$.
If $\beta_\infty < 1$, the limit of $D$ is in the singular set ${\cal S}$ of $M_\infty$ since $\omega_i$ converge to $\omega_\infty$ in the $C^\infty$-topology outside ${\cal S}$ as shown in last section. The difficulty lies in the case when $\beta_\infty \,=\,1$.
By \cite{tianwang}, ${\cal S}$ is of codimension at least $4$, so $M_\infty$ is actually smooth outside a closed subset of codimension $4$. Related results for smooth K\"ahler-Einstein metrics
were proved before (cf. \cite{cheegercoldingtian}, \cite{cheeger}). However, a priori, it is not even clear
if $\omega_i$ converge to $\omega_\infty$ in a stronger topology on any open subset of $M_\infty\backslash {\cal S}$. The original arguments in \cite{cheegercoldingtian} rely on an argument in \cite{anderson} which works only for smooth metrics. It fails for conic K\"ahler-Einstein metrics. So we need to have a new approach. In the course
of proving our main result in this section, we also exam the limit of $D$ in $M_\infty$.

First we describe a general and important construction: Given any conic metric $\omega$ with cone angle $2\pi \beta$ along $D$, its determinant gives a
Hermitian metric $\tilde H$ on $K_M^{-1}$ outside $D$. For simplicity, we will also denote by $\tilde H$ the induced Hermitian metric on $K_M^{-\ell}$ for any $\ell >0$.
However, $\tilde H$ is singular along $D$, more precisely, if $S$ is a defining section of $D$, then it is of the order $||S||_0^{-2(1-\beta)}$ along $D$, where
$||\cdot ||_0$ is a fixed Hermitian norm.
This implies that
$\tilde H(S,S)^{\frac{1-\beta}{\mu}} \tilde H$ is bounded along $D$, where $\mu = 1- (1-\beta) \lambda$.
On the other hand, there is a unique $f$ such that
as currents,
$${\rm Ric}(\omega)\,=\, \mu \,\omega\,+\,2\pi (1-\beta)\, [D]\,+ \,\sqrt{-1}\,\partial\bar\partial h,$$
where $f$ is normalized by
$$\int_M \left (e^h -1\right )\,\omega^n\,=\,0.$$
Note that $h$ is H\"older continuous. Put
$$H_\omega(\cdot,\cdot)\,=\, e^{\frac{h}{\mu}} \,\tilde H (S,S)^{\frac{1-\beta}{\mu}}\, \tilde H (\cdot,\cdot) ,$$
then as a current, the curvature of $H_\omega$ is equal to
$${\rm Ric}(\omega)\,-\, \frac{1-\beta}{\mu}\, \sqrt{-1}\, \partial\bar\partial \log \tilde H(S,S) \,- \,\frac{\sqrt{-1}}{\mu} \,\partial\bar\partial h\,=\, \omega.$$
Also we normalize $H_\omega$ by scaling $S$ such that
$$\int_M H_\omega (S,S) \,\omega^n\,=\,\int_M e^{\frac{\lambda h}{\mu}}\,\tilde H (S,S)^{\frac{1}{\mu}}\, \omega^n\,=\,1.$$
Such a Hermitian metric $H_\omega$ is uniquely determined by $\omega$ and $D$ and called the associated Hermitian metric of $\omega$.
If $\omega$ is conic K\"ahler-Einstein, its associated metric $H_\omega$ is determined by the volume form $\omega^n$, e.g., in local holomorphic coordinates
$z_1,\cdots,z_n$, write
$$\omega\,=\,\sqrt{-1}\,g_{i\bar j} \,dz_i\wedge d\bar z_j~~~{\rm and}~~~S \,=\, f \,
\left (\frac{\partial }{\partial z_1}\wedge\cdots \wedge \frac{\partial}{\partial z_n} \right)^\lambda,$$
then $H_\omega$ is represented by
$$\det (g_{i\bar j} )^{\frac{1}{\mu}}\, |f|^{\frac{2(1-\beta)}{\mu}}.$$
In particular, it implies that for any $\sigma\in H^0(M,K_M^{-\ell})$, $H_\omega(\sigma,\sigma)$ is bounded along $D$.


Now we recall some identities for pluri-anti-canonical sections.
\begin{lemm}
\label{lemm:3-2}
Let $\omega_i$ be as above and $H_i$ be the associated Hermitian metric on $K_M^{-1}$. Then for any $\sigma\in H^0(M, K_M^{-\ell})$, we have
(in the sense of distribution)
\begin{equation}
\label{eq:bochner-1}
\Delta_{i} ||\sigma||_i^2 \,=\, ||\nabla \sigma||_i^2 \,-\,n \ell \,||\sigma||_i^2
\end{equation}
and
\begin{equation}
\label{eq:bochner-2}
\Delta_{i} ||\nabla \sigma||_i^2 \,=\, ||\nabla^2 \sigma||_i^2 \,-\, ((n+2)\,\ell\,-\,\mu_i)\,||\nabla \sigma||_i^2,
\end{equation}
where $||\cdot||_i$ denotes the Hermitian norm on $K_M^{-\ell}$ induced by $H_i\,=\,H_{\omega_i}$, $\nabla$ denotes the covariant derivative of $H_i$ and $\Delta_i$ denotes the Laplacian of $\omega_i$.
\end{lemm}

\begin{proof}
On $M\backslash D$, both \eqref{eq:bochner-1} and \eqref{eq:bochner-2} were already derived in \cite{tian90} by direct computations.
Since $||\sigma||_i^2$ is bounded, \eqref{eq:bochner-1} holds on $M$.

By a direct computation in local coordinates, one can also show that $||\nabla\sigma||_i^2$ is bounded along $D$, so \eqref{eq:bochner-2} also holds.
\end{proof}

Applying the standard Moser iteration to \eqref{eq:bochner-1} and \eqref{eq:bochner-2} and using Lemma \ref{lemm:3-1}, we obtain
\begin{coro}
\label{coro:3-1}
There is a uniform constant $C$ such that for any $\sigma \in H^0(M, K_{M}^{-\ell})$, we have
\begin{equation}
\label{eq:esti-1}
\sup_{M}\left( ||\sigma||_i\,+\, \ell ^{-\frac{1}{2}}\,||\nabla \sigma||_i\right ) \,\le\, C\, \ell^{\frac{n}{2}} \,\left (\int _{M} ||\sigma||_i^2 \omega_i^n \right )^{\frac{1}{2}}.
\end{equation}
\end{coro}

If $\sigma_i$ is a sequence in $H^0(M, K_M^{-\ell})$ satisfying:
$$\int_M ||\sigma_i||_i^2\,\omega_i^n\,=\, 1,$$
then by Corollary \ref{coro:3-1}, $||\sigma_i||_i$ and their derivative are uniformly bounded. It implies that $||\sigma_i||_i$ are uniformly continuous. Hence,
by taking a subsequence if necessary,
we may assume $||\sigma_i||_i$ converge to a Lipschtz function $F_\infty$ as $i$ tends to $\infty$, moreover, we have
$$\int_{M_\infty} F_\infty^2 \,\omega^n_\infty\,=\,1.$$
In particular, $F_\infty$ is non-zero.

Now we assume $\sigma_i = a_i S $, where $a_i$ are constants and $S$ is a defining section of $D$.
Then $||\sigma_i||_i(x)\,=\,0$ if and only if $x\in D$
.
If $F_\infty(x) \not= 0$ for some $x\in M_\infty\backslash {\cal S}$, then for a sufficiently small $r > 0$, we have
$$ 2\, F_\infty (y) \,\ge \, F_\infty(x)\,>\,0, ~~\forall \,y\in B_r(x,\omega_\infty).$$
This is because $F_\infty$ is continuous. This implies
$$B_r(x,\omega_\infty)\subset M_\infty\backslash {\cal S}.$$
Since $||\sigma_i||_i$ converge to $F_\infty$ uniformly, for $i$ sufficiently large, $||\sigma_i||_i\,>\,0$
on those geodesic balls $B_r(x_i,\omega_i)$ of $(M,\omega_i)$
which converge to $B_r(x,\omega_\infty)$ in the Gromov-Hausdorff topology. It follows that $B_r(x_i,\omega_i) \subset M\backslash D$, that is, each
$B_r(x_i,\omega_i)$ lies in the smooth part of $(M,\omega_i)$. On the other hand, since $x$ is a smooth point of $M_\infty$, by choosing smaller $r$,
we can make the volume of $B_r(x_i,\omega_i)$ sufficiently close to that of corresponding Euclidean ball, then
as one argued in \cite{cheegercoldingtian} by a result of \cite{anderson}, $\omega_i$ restricted to
$B_r(x_i,\omega_i)$ converge to $\omega_\infty$ on any compact subset of $B_r(x,\omega_\infty)$ in the $C^\infty$-topology.
Thus, $\omega_i$ converge to $\omega_\infty$ in the $C^\infty$-topology on the non-empty open subset $M_\infty \backslash F_\infty^{-1}(0)\cup {\cal S}$.

Next we want to show that $F_\infty^{-1}(0)$ does not contain any open subset, or equivalently, $M_\infty \backslash F_\infty^{-1}(0)$ is an open-dense
subset in $M_\infty$. We prove it by contradiction.
If it is false, say $U\subset F^{-1}(0)$ is open, using the fact that $||\sigma_i||_i$ is uniformly bounded from above, we have
\begin{equation}
\label{eq:unbd}
\lim_{i\to \infty} \int_M \log (\frac{1}{i} + ||\sigma_i||_i^2) \,\omega_i^n \,=\,-\infty.
\end{equation}
By a direct computation, we have
$$\omega_i\,+\, \sqrt{-1} \partial\bar\partial \log (\frac{1}{i} + ||\sigma_i||_i^2) \,=\, \frac{\omega_i} {1 + i\, ||\sigma_i||_i^2}\,+\,\frac{i\,D\sigma_i\wedge \overline{D\sigma_i}}{(1 + i \,||\sigma_i||^2_i)^2}\,\ge\,0.$$
It implies
$$\Delta_i \log (\frac{1}{i} + ||\sigma_i||_i ^2) \,\ge\, - n.$$
Using the Sobolev inequality in Lemma \ref{lemm:3-1} and the Moser iteration, we can deduce
$$\sup_M \,\log (\frac{1}{i} + ||\sigma_i||_i^2) \,\le\, C\,\left ( 1\,+\,\int_M \log (\frac{1}{i} + ||\sigma_i||_i^2)\,\omega_i^n\right),$$
where $C$ is a uniform constant. By \eqref{eq:unbd},
$$\lim_{i\to \infty} \sup_M \,\log (\frac{1}{i} + ||\sigma_i||_i^2)\,=\,-\infty.$$
However, since the $L^2$-norm of $||\sigma_i||_i$ is equal to $1$, there is a constant $c$ independent of $i$ such that
$$\sup_M\, \log (\frac{1}{i} + ||\sigma_i||_i^2) \,\ge \, - c.$$
This leads to a contradiction. Therefore, $M_\infty \backslash F_\infty^{-1}(0)$ is dense.

By our definition of the metric $H_i$ associated to $\omega_i$, in local holomorphic coordinates $z_1,\cdots,z_n$ away from $D$, we have
$$ ||\sigma||_i^2\,=\, \left (\left(\det(g_{a\bar b})\right)^{\lambda } |w|^2\right )^{\frac{1}{\mu}}$$
where
$$\sigma \,=\,w \,\frac{\partial}{\partial z_1} \wedge \cdots\wedge
\frac{\partial}{\partial z_1}~~~{\rm and} ~~~\omega_i \,=\, \sqrt{-1} \,\sum_{a,b=1}^n \,g_{a\bar b}\, dz_a\wedge d\bar z_b\,.$$
Since $\omega_i$ converge to $\omega_\infty$ in the $C^\infty$-topology outside ${\cal S}$, it follows from the above that
$\sigma_i$ converge to a holomorphic section $\sigma_\infty$ on $M\backslash F_\infty^{-1}(0)\cup {\cal S}$.\footnote{The singular set ${\cal S}$ may overlap with $F^{-1}_\infty(0)$ along a subset of complex codimension $1$.}
Clearly,
$\sigma_\infty$ is bounded with respect to the Hermitian metric associated to $\omega_\infty$, so it extends to a holomorphic section of $K_{M_\infty}^{-\lambda}$
on the regular part $M\backslash {\cal S}$.

Next we exam the limit of $D$ under the convergence of $(M,\omega_i)$. Since $||\sigma_i||_i\,=\,0$ on $D$, the limit of $D$ must lie in $D_\infty$, where $D_\infty$ denotes the zero set of $F_\infty$. We claim that the limit of $D$ coincides with $D_\infty$. If this is not true, there are $x\in D_\infty$ and $r>0$ such that
$B_{2r}(x, d_\infty)\cap D_\infty$ is disjoint from the limit of $D$. Choose $x_i\in M$ go to $x$ as $(M,\omega_i)$ converge to $(M_\infty, d_\infty)$, then for
$i$ sufficiently large, $B_r(x_i, \omega_i)$ is disjoint from $D$, so lies in the smooth part of $(M,\omega_i)$. The regularity theory in \cite{cheegercoldingtian}
implies that ${\cal S}\cap B_r(x,d_\infty)$ is of complex codimension at least $2$ and near a generic point $y \in B_{r}(x, d_\infty)\cap D_\infty$, $\sigma_\infty$ is holomorphic and defines $D_\infty$, moreover, the convergence of $(M,\omega_i)$ to $(M_\infty,d_\infty)$ is in $C^\infty$-topology and $\sigma_i$ converge
to $\sigma_\infty$ near $y$ , so $\sigma_i$ must vanish somewhere in $B_r(x_i,\omega_i)$, a contradiction. This shows that the limit of $D$ coincides with $D_\infty$.

If $\beta_\infty \,=\,1$, the singular set ${\cal S}$ is of complex dimension at least $2$ and $\sigma_\infty\in H^0(M_\infty, K_{M_\infty}^{-\lambda})$ which consists of all holomorphic sections of $K_{M_\infty}^{-\lambda}$ on $M_\infty\backslash {\cal S}$.
Then $D_\infty$ is simply the divisor $\{\sigma_\infty \,=\,0\}$.

Summarizing the above discussions, we have

\begin{theo}
\label{th:3-2}
Let $(M_\infty,\omega_\infty)$, ${\cal S}$ etc. be as in Theorem \ref{th:3-1}. Then
$(M,\omega_i)$ converge to $(M_\infty,\omega_\infty)$ in the $C^\infty$-topology outside $\bar {\cal S}\cup D_\infty$ for a closed subset
$\bar {\cal S}$ of codimension at least $4$ and $D$ converges to $D_\infty$ in the Gromov-Hausdorff topology.
If $\beta_\infty \,<\,1$, ${\cal S}\,=\,\bar {\cal S} \cup D_\infty$. If $\beta_\infty \,=\,1$, ${\cal S}\,=\,\bar {\cal S}$ and $D_\infty$ is
a divisor of $K_{M_\infty}^{-\lambda}$.\footnote{It follows from the partial $C^0$-estimate in the next section that the same holds even if $\beta_\infty \,<\,1$.}
\end{theo}

\section{Partial $C^0$-estimate}

In this section, we prove Theorem \ref{th:main-1}. By our results on compactness of conic K\"ahler-Einstein metrics in last two sections, we need to prove only
the following:

\begin{theo}
\lab{th:4-1}
Let $M$ be a Fano manifold $M$ and $D$ be a smooth divisor whose Poincare dual is $\lambda\, c_1(M)$. Let
$\omega_i$ be a sequence of conic K\"ahler-Einstein metrics on $M$ with conic angle $2\pi \beta_i$ along $D$ satisfying:
$$\lim \beta_i \,=\,\beta_\infty > 0~~~{\rm and}~~~0\,<\,(1-\beta_\infty)\,\lambda\,<\, 1.$$ 
We also assume that $(M,\omega_i)$
converge to a (possibly singular) conic K\"ahler-Einstein manifold $(M_\infty,\omega_\infty)$ as described in Theorem \ref{th:3-2}.
Then there are uniform constants
$c_k=c(k, n, \lambda,\beta_\infty)>0$ for $k\ge 1$ and $\ell_a \to \infty$
such that for $\ell = \ell_a$,

\begin{equation}
\lab{eq: conicpartial}
\rho_{\omega_i, \ell} \,\ge\, c_\ell\,>\, 0.
\end{equation}
\end{theo}

For the readers' convenience, we recall the definition of $\rho_{\omega_i, \ell}$: Let $H_i$ be the Hermitian metric on $K_M^{-1}$ associated to $\omega_i$,
then for any orthonormal basis  $\{S_\alpha\}_{0\le \alpha\le N}$ of $H^0(M, K_M^{-\ell})$ with respect to the inner product induced by $H_i$ and $\omega_i$, we have
\begin{equation}
\lab{eq:4-1}
\rho_{\omega_i,\ell} (x) \,=\,\sum_{\alpha=0}^N\, H_i(S_\alpha,S_\alpha) (x),
\end{equation}

We have shown in last section that the defining sections $\sigma_i$ of $D$ normalized with respect to $H_i$ converge to a holomorphic section $\sigma_\infty$
of $K_{M_\infty}^{-\lambda} $ on $M\backslash {\cal S}$ satisfying: In any local coordinates $z_1,\cdots,z_n$ outside ${\cal S}$, we have
\begin{equation}
\label{eq:4-1}
\left(\det(g_{a\bar b})\right)^{\lambda } |w|^2 \,< \, \infty
\end{equation}
where
$$\sigma_\infty \,=\,w \,\frac{\partial}{\partial z_1} \wedge \cdots\wedge
\frac{\partial}{\partial z_1}~~~{\rm and} ~~~\omega_\infty \,=\, \sqrt{-1}  \,g_{a\bar b}\, dz_a\wedge d\bar z_b\,.$$
Define a Hermitian metric $H_\infty$ on $K_{M_\infty}^{-1}$ on $M_\infty\backslash {\cal S}$ by
\begin{equation}
\label{eq:4-2}
H_\infty\,=\, \tilde H_\infty (\sigma_\infty,\sigma_\infty )^{\frac{1-\beta}{\mu}}\, \tilde H_\infty.
\end{equation}
Here $\tilde H_\infty$ denotes the Hermitian metric induced by the determinant of $\omega_\infty$. The following can be easily proved.

\begin{lemm}
\label{lemm:4-1}
The Hermitian metrics $H_i$ converge to $H_\infty$ on $M_\infty\backslash {\cal S}$ in the $C^\infty$-topology. Moreover,
we have
$$ H_\infty (\sigma_\infty,\sigma_\infty)\,<\, \infty~~~{\rm and}~~~\int_{M_\infty} H_\infty (\sigma_\infty,\sigma_\infty)\,\omega_\infty^n \,=\, 1.$$
\end{lemm}

Let us first specify the holomorphic sections of $K_{M_\infty}^{-\ell}$ we will use here.
\footnote{This is needed since $M_\infty$ may have singularity along a subset
${\cal S}$ of complex codimension $1$. However, we will prove later that $M_\infty$ is actually smooth outside a subset of complex codimension at least $2$.}
By a holomorphic section of $K_{M_\infty}^{-\ell}$ on $M_\infty$ ($\ell > 0$), we mean a holomorphic section $\sigma$ of $K_{M_\infty}^{-\ell}$
on $M_\infty\backslash {\cal S}$ with $ H_\infty (\sigma,\sigma)$ bounded.

We denote by $H^0(M_\infty, K_{M_\infty}^{-\ell})$ the space of all holomorphic sections of $K_{M_\infty}^{-\ell}$
on $M$. If $M_\infty$ is smooth outside a closed subset of codimension $4$, then it
coincides with the definition we used in literature.

\begin{lemm}
\label{lemm:conv}
For any fixed $\ell \,>\,0$, if $\{\tau_i\}$ is any sequence of $H^0(M,K_M^{-\ell})$ satisfying:
$$\int_M H_i(\tau_i,\tau_i)\,\omega_i^n \,=\,1,$$
then a subsequence of $\tau_i$ converges to a section $\tau_\infty $ in $H^0(M_\infty, K_{M_\infty}^{-\ell})$.

\end{lemm}

This follows from the estimate in Corollary \ref{coro:3-1} and standard arguments. It implies that $\rho_{\omega_i,\ell}$ are uniformly continuous,
in particular, they converge to a continuous function on $M_\infty$. This function is actually equal to $\rho_{\omega_\infty,\ell}$ as shown in the end of this section, but we do not need this to prove Theorem \ref{th:4-1}.

Thus, in order to prove Theorem \ref{th:4-1}, we only need to show that for a sequence of $\ell$,
\begin{equation}
\label{eq:4-5}
\inf_i \inf_x \rho_{\omega_i,\ell} (x)\, >\, 0.
\end{equation}

Since $\rho_{\omega_i,\ell} $ are uniformly continuous and $M_\infty$ is compact, it suffices
to show that for any $x\in M_\infty$, there is an
$\ell$ and sequence $x_i\in M$ such that $\lim x_i\,=\,x$ and
\begin{equation}
\label{eq:4-6}
\inf_i \rho_{\omega_i,\ell} (x_i)\, > \,0.
\end{equation}

The following lemma provides the $L^2$-estimate for $\bar\partial $-operator on $(M,\omega_i)$. It can be proved by using the smooth approximations
$\tilde\omega_i$ of $\omega_i$
with Ricci curvature bounded from below.
\begin{lemm}
\label{lemm:4-2}
For any $\ell > 0$, if $\zeta$ is a (0,1)-form with values in $K_M^{-\ell}$ and $\bar\partial \zeta = 0$, there is a smooth section $\vartheta$ of $K_M^{-\ell}$ such that
$\bar \partial \vartheta = \zeta$ and
$$\int _M ||\vartheta||_i^2 \,\omega_i^n \,\le \, \frac{1}{\ell + \mu}\,\int_M ||\zeta||_i^2\,\omega_i^n,$$
where $||\cdot||_i$ denotes the norm induced by $H_i$ and $\omega_i$.

\end{lemm}

We have seen that for any $r_j\mapsto 0$, by taking a subsequence if necessary,
we have a tangent cone ${\cal C}_x$ of $(M_\infty,\omega_\infty)$ at $x$, which is the limit of $(M_\infty, r_j^{-2} \omega_\infty, x)$ in the Gromov-Hausdorff topology, satisfying:

\vskip 0.1in
\noindent
${\bf T}_1$. Each ${\cal C}_x$ is regular outside a closed subcone ${\cal S}_x$ of complex codimension at least $1$. Such a ${\cal S}_x$ is the singular set of ${\cal C}_x$;

\vskip 0.1in
\noindent
${\bf T}_2$. There is an natural K\"ahler Ricci-flat metric $g_x$ on ${\cal C}_x \backslash {\cal S}_x$ which is also a cone metric. Its K\"ahler form $\omega_x$ is equal to $\sqrt{-1} \,\partial\bar\partial \rho_x^2$ on the regular part of ${\cal C}_x$, where $\rho_x$ denotes the distance function from the vertex of ${\cal C}_x$, denoted by $x$ for simplicity.
\vskip 0.1in
We will denote by $L_x$ the trivial bundle ${\cal C}_x\times \CC$ over ${\cal C}_x$ equipped with
the Hermitian metric $e^{-\rho_x^2}\,|\cdot |^2$. The curvature of this Hermitian metric is given by $\omega_x$.

As before, we denote by ${\cal S}_k$ ($k=0,1,\cdots,2n-1$) the subset of $M_\infty$
consisting of points for which no tangent cone splits off a factor, $\RR^{k+1}$, isometrically. Clearly,
${\cal S}_0 \subset {\cal S}_1\subset \cdots \subset {\cal S}_{2n-1}$. It is proved by Cheeger-Colding that ${\cal S}_{2n-1} = \emptyset$, $\dim {\cal S}_k \le k$ and ${\cal S}\,=\,{\cal S}_{2n-2}$.

The following lemma can be proved by using the slicing arguments in \cite{cheegercoldingtian} and the fact that $(M_\infty, \omega_\infty)$ is the limit of
conic K\"ahler-Einstein metrics $(M,\omega_i)$ with cone angle along $2\pi \beta_i$ along $D$.
\begin{lemm}
\label{lemm:4-3} For any $x\,\in\, {\cal S}_{2n-2}\backslash \bigcup_{k< 2n-2} {\cal S}_k$, if ${\cal C}_x \,=\, \CC ^{n-1}\times {\cal C}_x'$, then
$g_x $ is a product of the Euclidean metric on $\CC^{n-1}$ with a flat conic metric
on ${\cal C}_x'$, which is biholomorphic to $\CC$, of angle $2\pi \mu_a$ ($a =1,\cdots, l$), where $\bar\mu=\mu_a$ is given as in Theorem \ref{th:tangentcone}. Moreover,
for any $x\,\in\, {\cal S}\,\subset\, M_\infty$, if ${\cal S}_x$ is of complex codimension $1$, then there is a closed subcone $\bar {\cal S}_x\subset {\cal S}_x$ of complex codimension
at least $2$ such that $g_x$ is asymptotic to the product metric described above at any $y\in {\cal S}_x\backslash \bar {\cal S}_x$, i.e., a tangent cone of $({\cal C}_x,g_x)$
at $y$ is isometric to a product of the Euclidean metric on $\CC^{n-1}$ with a conic metric
on ${\cal C}_x'$ of angle $2 \pi \mu_a\,<\,2\pi$.

\end{lemm}

\begin{rema}\label{rema:angle}
As we said after Theorem \ref{th:tangentcone}, by the volume comparison, we know $\bar\beta_\infty\,\le\,\mu_a\,\le\, \beta_\infty$ for some $\bar\beta_\infty$ 
depending only on the diameter and volume of $(M_\infty,\omega_\infty)$. However, in our proof, we may assume that $\beta_\infty \ge 1 - \lambda^{-1} + \epsilon$ 
for some $\epsilon > 0$, so $\bar\beta_\infty$ can be assume to be uniform. When $\beta_\infty=1$, all $\mu_a = 1$, so there is only one. If $\beta_\infty < 1$,
since $(1-\mu_a) = m_a (1-\beta_\infty)$ for some integer $m_a$, there is a bound on $l$ as well. In fact, one should be able to prove that there is a uniform bound
on $l$ depending only on $\lambda$. 
\end{rema}

Without loss of generality, in the following, for each $j$, we set $k_j$ to be the integral part of $r_j^{-2}$.

Now we fix some notations: For any $ \epsilon \,>\,0$, we put
$$V(x; \epsilon)\,=\,\{ \,y \,\in \, {\cal C}_x \,|\, y\,\in\, B_{\epsilon^{-1}}(0,g_x)\,\backslash \,\overline{B_{\epsilon}(0,g_x)},\,\, d(y, {\cal S}_x )\, > \,\epsilon\,\,\},$$
where $B_R(o,g_x)$ denotes the geodesic ball of $({\cal C}_x, g_x)$ centered at the vertex and with radius $R$.

If ${\cal C}_x$ has isolated singularity, then ${\cal S}_x\,=\,\{o\}$ and
$$V(x;\epsilon)\,=\, \{\, y\,\in\, {\cal C}_x \,|\, y\,\in\, B_{\epsilon^{-1}}(0,g_x)\,\backslash \,\overline{B_{\epsilon}(0,g_x)}\,\,\}.
$$
Let $r_j^{-2}$ be the above sequence such that $(M_\infty, r^{-2}_j \omega_\infty, x)$ converges to $({\cal C}_x, g_x, o)$.
By \cite{cheegercoldingtian},
for any $\epsilon > 0$, whenever $i$ is sufficiently large, there are diffeomorphisms $\phi_j: V(x;\epsilon)\mapsto  M_\infty\backslash {\cal S}$,
where ${\cal S}$ is the singular set of $M_\infty$,
satisfying:
\vskip 0.1in
\noindent
(1) $d(x, \phi_j(V(x; \epsilon))) \,<\, 10 \epsilon r_j$ and $\phi_j(V(x;\epsilon)) \subset B_{(1+\epsilon^{-1}) r_j}(x)$, where
$B_R(x)$ the geodesic ball of $(M_\infty, \omega_\infty)$ with radius $R$ and center at $x$;
\vskip 0.1in
\noindent
(2) If $g_\infty$ is the K\"ahler metric with the K\"ahler form $\omega_\infty$ on $M_\infty\backslash {\cal S}$, then
\begin{equation}
\label{eq: bound-1}
\lim_{j\to \infty} ||r_j^{-2} \phi_j^*g_\infty - g_x ||_{C^6(V(x; \frac{\epsilon}{2}))} \,=\,0,
\end{equation}
where the norm is defined in terms of the metric $g_x$.

\begin{lemm}
\label{lemm:4-4}
For any $\delta$ sufficiently small, there are a sufficiently large $\ell \,=\,k_j$ and an isomorphism $\psi $ from the trivial bundle ${\cal C}_x\times \CC$ onto
$K_{M_\infty}^{- \ell}$
over $V(x;\epsilon)$ commuting with $\phi\,=\,\phi_j$ satisfying:
\begin{equation}
\label{eq:est-2}
||\psi(1)||^{2} \,=\, e^{-\rho_x^2} ~~~{\rm and}~~~ ||\nabla \psi ||_{C^4(V(x; \epsilon))} \,\le \, \delta,
\end{equation}
where $||\cdot||^2$ denotes the induced norm on $K_{M_\infty}^{-\ell}$ by $\omega_\infty$, $\nabla$ denotes the covariant derivative with respect to the norms
$||\cdot||^2$ and $e^{-\rho_x^2}\, |\cdot |^2$.
\end{lemm}

\begin{proof} The arguments of its proof are pretty standard, so we just outline it.
%
%
%
We cover $V(x; \epsilon)$ by finitely many geodesic balls $B_{s_\alpha}(y_\alpha)$ ($1\le \alpha \le N$) satisfying:

\vskip 0.1in
\noindent
(i) The closure of each $B_{2s_\alpha}(y_\alpha)$ is strongly convex and contained in ${\rm Reg}({\cal C}_x)$;.

\vskip 0.1in
\noindent
(ii) The half balls $B_{s_\alpha/2}(y_\alpha)$ are mutually disjoint.
\vskip 0.1in

Now we choose $\ell\,=\,\ell_j$ sufficiently large and construct $\psi$.

First we construct $\tilde \psi_\alpha$ over each $B_{2 s_\alpha}(y_\alpha)$.
For any $y\in B_{2 s_\alpha}(y_\alpha)$, let $\gamma_y\subset B_{2 s_\alpha}(y_\alpha)$ be the unique minimizing geodesic from
$y_\alpha$ to $y$. We define $\tilde\psi_\alpha$ as follows:
First we define
$\tilde \psi_\alpha (1)\in L|_{\phi(y_\alpha)}$ such that
$$||\tilde \psi_\alpha(1)||^2 \,=\, e^{-\rho_x^2(y_\alpha)},$$
where $L\,=\,K_{M_\infty}^{-\ell}$. Next, for any $y\in U_\alpha$, where $U_\alpha\,=\,B_{2 s_\alpha}(y_\alpha)$,
define
$$\tilde \psi_\alpha: \CC \mapsto L|_y,~~~\tilde \psi_\alpha (a(y)) = \tau(\phi(y)),$$
where $a(y)$ is the parallel transport of $1$ along $\gamma_y$ with respect to the norm $e^{-\rho_x^2}\, |\cdot |^2$ and $\tau(\phi(y))$ is
the parallel transport of $\psi(1)$ along $\phi\circ\gamma_y$ with respect to the norm $||\cdot||^2$.

Clearly, we have the first equation in \eqref{eq:est-2}. The estimates on derivatives can be done as follows: If $a: U_\alpha \mapsto U_\alpha\times \CC$ and $\tau: U_\alpha \mapsto \phi^*L|_{U_\alpha}$ are two sections such that $\tilde \psi_\alpha (a)\, =\, \tau$, then we have
the identity:
$$\nabla \tau \,=\, \nabla\tilde \psi_\alpha (a) \,+\, \tilde\psi_\alpha (\nabla a),$$
where $\nabla$ denote the covariant derivatives with respect to the given norms on line bundles ${\cal C}_x\times \CC$ and $L$.
By the definition, one can easily see that $\nabla \tilde\psi_\alpha (y_\alpha) \equiv 0$. To estimate $\nabla\tilde\psi_\alpha$ at $y$,
we differentiate along $\gamma_y$ to get
$$\nabla_T\nabla_X \tau \,=\, \nabla_T(\nabla_X\tilde\psi_\alpha (a)) + \tilde\psi_\alpha (\nabla_T \nabla_X a),$$
where $T$ is the unit tangent of $\gamma_y$ and $X$ is a vector field along $\gamma_y$ with $[T,X]=0$.
Here we have used the fact that $\nabla_T\tilde\psi_\alpha =0$ which follows from the definition.
Using the curvature formula, we see that it is the same as
$$k \phi^*\omega_\infty (T,X)\,\tilde\psi_\alpha (a) \,=\, \nabla_T(\nabla_X\tilde \psi_\alpha (a)) \,+\, \omega_x(T,X)\, a.$$
Using the fact that $\omega_x$ is the limit of $k \phi^* \omega_\infty$,
we can deduce from the above that $\nabla_T(\nabla_X\tilde \psi_\alpha(a))$ converges to $0$ as $i$ tends to $\infty$. Since $\nabla_X\tilde \psi_\alpha = 0$ at $y_\alpha$,
we see that $||\nabla \tilde \psi_\alpha ||_{C^0(U_\alpha)}$ can be made sufficiently small. The higher derivatives can be bounded in a similar way.

Next we want to modify each $\tilde\psi_\alpha$. For any $\alpha, \beta$, we set
$$\theta_{\alpha\gamma} \,=\, \tilde\psi_\alpha^{-1}\circ \tilde\psi_\gamma: U_\alpha\cap U_\gamma \mapsto S^1.$$
Clearly, we have
$$\theta_{\alpha\kappa} \,=\,\theta_{\alpha\gamma}\cdot\theta_{\gamma\kappa}~~~{\rm on}~~U_\alpha\cap U_\gamma\cap U_\kappa,$$
so we have a closed cycle $\{\theta_{\alpha\gamma}\}$. By the derivative estimates on each $\tilde\psi_\alpha$, we know that each $\theta_{\alpha\gamma}$ is close to a constant. Therefore, we
can modify $\tilde\psi_\alpha$'s such that each transition function $\theta_{\alpha\gamma}$ is a unit constant, that is, we can construct $\zeta_\alpha:
U_\alpha \mapsto S^1$ such that if we replace each $\tilde\psi_\alpha$ by $\tilde \psi_\alpha\cdot \zeta_\alpha$, the corresponding transition functions
are constant. Moreover we can dominate $||\nabla \zeta_\alpha ||_{C^3}$ by the norm $||\nabla \tilde \psi_\alpha ||_{C^3}$ (possibly) on a slightly larger ball.

The cycle $\{\theta_{\alpha\gamma}\}$ of constants gives rise to a flat bundle $F$, and we have constructed
an isomorphism
$$\xi: F\mapsto K_{M_\infty}^{-\ell}$$
over an neighborhood of $\overline{V(x;\epsilon)}$ satisfying all the estimates in \eqref{eq:est-2}.

If we replace $\ell $ by $k \ell  $, we get an analogous isomorphism
$$\xi: F^k  \mapsto K_{M_\infty}^{-k \ell}.$$
Since the flat bundle $F$ is given by a representation
$$\rho:\pi_1(V(x;\epsilon))\mapsto S^1, $$
there is an $k$ such that $F^k$ is essentially trivial, i.e., the corresponding transition functions are
in a neighborhood of the identity in $S^1$. Then we can further modify $\tilde\psi_\alpha$ such that $\theta_{\alpha\gamma}\,=\,1$ for any $\alpha$ and $\gamma$.
So we can get the required $\psi$ by setting $\psi \,=\, \tilde\psi_\alpha $ on $V(x;\epsilon)\cap B_{s_\alpha}(x_\alpha)$.

In fact, one can show that either of the following conditions holds for :

\vskip 0.1in
\noindent
(1) There is a tangent cone ${\cal C}_x$ of the form $\CC^{n-1}\times {\cal C}_x'$ for a 2-dimensional flat cone ${\cal C}_x'$ of angle $2\pi \mu_a$, where
$\mu_a$ are given in Lemma \ref{lemm:4-3} for $a =1,\cdots, l$;

\vskip 0.1in
\noindent
(2) There is a closed subcone $\bar {\cal S}_x\subset {\cal S}_x$ of codimension at least $4$ such that for every $y\in {\cal S}_x\backslash \bar {\cal S}_x$,
any tangent cone ${\cal C}_y $ of ${\cal C}_x$ at $y$ is of the form $\CC^{n-1}\times {\cal C}_y'$ for a 2-dimensional flat cone ${\cal C}_y'$ of angle $2\pi \mu_a$, where
$\mu_a$ are given in Lemma \ref{lemm:4-3} for $a =1,\cdots, l$.
Moreover, ${\cal C}_x\backslash \bar {\cal S}_x$ has finite fundamental group of order $\nu \ge 1$.

\vskip 0.1in
Thus we just need to take $\ell$ to be a multiple of $\nu$ such that for $a=1,\cdots,l$,
$\ell \mu_a$ are sufficiently close to $1$ modulo $\ZZ$ in the above construction of $\psi$. Since $\mu_a = 1 - m + m \beta_\infty$ for some integer,
the second condition is the same as requiring that $\ell \beta_\infty$ are sufficiently close to $1$ modulo $\ZZ$. 
\end{proof}

As for smooth K\"ahler-Einstein metrics, we will apply the $L^2$-estimate to proving \eqref{eq:4-6}, consequently, the partial $C^0$-estimate for
conic K\"ahler-Einstein metrics. The method is standard and resembles the one we used for Del-Pezzo surfaces in \cite{tian89}. First we construct
an approximated holomorphic section $\tilde\tau$ on $M_\infty$, then one can perturb it into a holomorphic section $\tau$ by the $L^2$-estimate for
$\bar\partial$-operators, finally, one uses the derivative estimate in Corollary \ref{coro:3-1} to conclude that $\tau(x)\,\not=\,0$.


Let $\epsilon\,>\,0$ and $\delta\,>\,0$ be sufficiently small and be determined later.
We fix $\ell$ to be the integral part of $r^{-2}$ and $r\,=\,r_j$ for a sufficiently large $j$ which may depend
on $\epsilon$ and $\delta$. Choose $\phi$ and $\psi$ by Lemma \ref{lemm:4-4}, then
there is a section
$\tau \,=\, \psi( 1 )$ of $K_{M_\infty}^{-\ell}$ on $\phi(V(x; \epsilon))$ satisfying:
$$||\tau||^2\,=\,e^{-\rho_x^2}.$$
By Lemma \ref{lemm:4-4}, for some uniform constant $C$, we have
$$||\bar\partial \tau||\, \le C \,\delta . $$

Now let us state a technical lemma.
\begin{lemm}
\label{lemm:cut-off-2}
For any $\bar \epsilon \,>\,0$, there is a smooth function $\gamma_{\bar\epsilon}$ on ${\cal C}_x$ satisfying:

\vskip 0.1in
\noindent
(1) $\gamma_{\bar\epsilon }(y)\,=\,1$ for any $y$ with $d(y,{\cal S}_x)\,\ge\,\bar\epsilon$, where $d(\cdot,\cdot)$ is the distance of $({\cal C}_x,g_x)$ ;

\vskip 0.1in
\noindent
(2) $0\,\le\,\gamma_{\bar\epsilon} \,\le\,1$ and $\gamma_{\bar\epsilon} (y)\,=\,0$ in an neighborhood of ${\cal S}_x$;

\vskip 0.1in
\noindent
(3) $|\nabla \gamma_{\bar\epsilon}|\,\le\, C$ for some constant $C\,=\,C(\bar\epsilon)$ and
$$\int_{B_{{\bar\epsilon}^{-1}}(o,g_x)} \,|\nabla \gamma_{\bar\epsilon}|^2\, \omega_x^n\,\le\,\bar\epsilon.$$
\end{lemm}
\begin{proof}
This is rather standard and has been known to me for quite a while. This is based on the fact that the Poincare metric on a punctured disc has finite volume.

First we consider the simplest case that ${\cal S}_x\,=\,\CC^{n-1}$, i.e., ${\cal C}_x$ is of the form $\CC^{n-1}\times {\cal C}_x'$, where ${\cal C}_x'$ is biholomorphic to $\CC$.
Moreover, the cone metric $g_x$ coincides with the standard cone metric
$$ g_{\bar\beta}\,=\,\sum _{i=1}^{n-1} dz_id\bar z_i\,+\, (d\rho^2 \,+\, \bar\beta^2 \rho^2 d\theta^2),$$
where $z_1,\cdots,z_{n-1}$ are coordinates of $\CC^{n-1}$ and $\bar\beta$ is one of $\mu_a$ given in Lemma \ref{lemm:4-3}. 
Clearly, $\rho \,=\, d(y,{\cal S}_x)$.

We denote by $\eta$ a cut-off function: $\RR \mapsto \RR$ satisfying:
$0\,\le\, \eta \,\le\, 1$, $|\eta'(t)| \,\le\, 1$ and
$$\eta(t)\,=\, 0~~{\rm for}~~t \,>\, \log(-\log \delta^3) ~~{\rm and}~~\eta(t) \,=\, 1 ~~{\rm for}~~ t \,< \,\log (-\log \delta) . $$
Here $\delta \,<\, 1/3$ is to be determined.
Now we define as follows: If $\rho(y) \,\ge\, \bar\epsilon /3$, put $\gamma_{\bar\epsilon}(y)\,=\,1$ and if $\rho(y) \,< \, \bar\epsilon $
$$\gamma_{\bar\epsilon}(y)\,=\, \eta \left ( \log\left (-\log \left(\frac{\rho(y)}{\bar\epsilon} \right)\right)\right).$$
Clearly, $\gamma_{\bar\epsilon}$ is a smooth function and we have
$$\gamma_{\bar\epsilon}(y)\,=\,1~~{\rm if}~~\rho(y) \,\ge\, \frac{\bar\epsilon}{3}~~~{\rm and}~~~\gamma_{\bar\epsilon}(y)\,=\,0~~{\rm if}~~\rho(y) \,\le\, \delta^3\bar\epsilon.$$
Furthermore, the support of $|\nabla \gamma_{\bar\epsilon}|(y)\,=\,0$ is contained in the region where
$\delta^3\bar\epsilon\,<\,\rho(y) \,< \,\delta \bar \epsilon$. In the region, we have
$$|\nabla \gamma_{\bar\epsilon}|\,\le\, \frac{1}{\rho (-\log \frac{\rho}{\bar \epsilon})}.$$
It follows that
$$\int_{B_{{\bar\epsilon}^{-1}}(o,g_x)} \,|\nabla \gamma_{\bar\epsilon}|^2\, \omega_x^n\,\le\,\frac{a_{n-1}}{ \bar\epsilon^{2n-2}} \int_{\delta^3}^{\delta}
\frac{dr}{r (-\log r)^2}\,\le\, \frac{a_{n-1} }{\bar\epsilon^{2n-2} (-\log \delta)},$$
where $a_{n-1}$ denotes the volume of the unit ball in $\RR^{2n-2}$.

Now choose $\delta$ such that $a_{n-1}\,\le\,\bar\epsilon^{2n-1} (-\log \delta)$, then we have
$$\int_{B_{{\bar\epsilon}^{-1}}(o,g_x)} \,|\nabla \gamma_{\bar\epsilon}|^2\,\le\,\bar\epsilon.$$
Clearly, we also have $|\nabla \gamma_{\bar\epsilon}|\,\le\,C$ for some $C\,=\,C(\bar\epsilon)$.

In general, as we have shown in Section 3 by using the arguments of \cite{cheegercoldingtian}, 
${\cal S}_x$ is a union of ${\cal S}_x^0$ and $\bar {\cal  S}_x$, where
$\bar {\cal S}_x$ is a closed subcone
and ${\cal S}^0_x$ is an open subcone of ${\cal S}_x$ which consists of all $y\in {\cal S}_x$ such that a tangent cone of $({\cal C}_x,g_x)$
at $y$ is isometric to $\CC^{n-1}\times {\cal C}_y'$ with the standard metric $g_{\bar\beta}$, 
where $(1-\bar\beta) = k ( 1-\beta_\infty)$ for some integer $k$.
Furthermore, $\bar {\cal S}_x$ is of complex codimension at least $2$.

We expect the following:
\vskip 0.1in
\noindent
${\bf A}_1$ ${\cal C}_x$ is a variety near ${\cal S}_x^0$ and ${\cal S}^0_x$ is a subvariety.
\vskip 0.1in
This can be proved by establishing a local version of Theorem \ref{th:conj-2} and by using the simplest case of
Lemma \ref{lemm:cut-off-2}. We refer the readers to Remark \ref{rema:c-x} for more discussions. 
Now we explain how to derive Lemma \ref{lemm:cut-off-2} under Assumption ${\bf A}_1$. 
This is intended for illustrating the idea of the proof of Lemma \ref{lemm:cut-off-2} before 
getting too tedious arguments based on known techniques.
We will complete the proof of Lemma \ref{lemm:cut-off-2}
by using an analogous, but weaker, version of Assumption ${\bf A}_1$ in Appendix.

Clearly, ${\bf A}_1$ implies the following;
\vskip 0.1in
\noindent
${\bf A}_1'$. ${\cal S}_x$ can be written as a union of two subcones ${\cal S}_{x,1}$ and ${\cal S}_{x,2}$ such that
${\cal S}_{x,2}$ is a closed subcone of complex codimension at least $2$ and
${\cal C}_x$ is smooth near ${\cal S}_{x,1}$ which is a smooth divisor.
\vskip 0.1in

For any small $\epsilon_0 >0$, since ${\cal S}_{x,2}$ has vanishing Hausdorff measure of dimension strictly bigger than $2n-4$, we can find a finite cover of ${\cal S}_{x,2} \cap B_{\bar\epsilon^{-1}}(x,g_x)$ by balls $B_{ r_a}(y_a,g_x)$ ($a\,=\,1,\cdots,l$)
satisfying:

\vskip 0.1in
\noindent
(i) $y_a \in {\cal S}_{x,2}$ and $2 r_a \le \epsilon_0$;

\vskip 0.1in
\noindent
(ii) $B_{r_a/2}(y_a,g_x)$ are mutually disjoint;

\vskip 0.1in
\noindent
(iii) $\sum_a r_a^{2n-3} \,\le\, 1$;

\vskip 0.1in
\noindent
(iv) The number of overlapping balls $B_{2r_a}(y_a,g_x)$ is uniformly bounded.
\vskip 0.1in

We denote by $\bar\eta$ a cut-off function:  $\RR \mapsto \RR$
satisfying: $0\,\le\, \bar \eta \,\le\, 1$, $|\bar \eta'(t)| \,\le\, 2$ and
$$\bar\eta(t)\,=\, 1~~{\rm for}~~t \,>\, \frac{3}{2} ~~{\rm and}~~\bar\eta(t) \,=\, 0 ~~{\rm for}~~ t \,\le \,1. $$

Put
$$\chi_a(y)\,=\,\bar\eta \left(\frac{d(y , y_a)}{r_a}\right)~~{\rm if}~y\in B_{2 r_a}(y_a,g_x)~~~{\rm and}~~~\chi_a(y)\,=\,1~~{\rm otherwise}.$$
Clearly, $\chi_a\equiv 0$ on $B_{r_a}(y_a,g_x)$. By (iv), near any $y$, the number of non-vanishing $\chi_a$ is uniformly bounded by $A$, so the product
function $\chi\,=\,\prod_a \chi_a$ is smooth and vanishes near ${\cal S}_{x,2}\cap B_{\bar\epsilon^{-1}}(x,g_x)$, furthermore, we have
\begin{equation}
\label{eq:prod-1}
\int_{{\cal C}_x} |\nabla \chi|^2 \, \omega_x^n \,\le \,A\,\sum_a \int_{B_{r_a}(y_a,g_x)} |\nabla \chi_a|^2 \, \omega_x^n\,\le\, C \,\epsilon_0,
\end{equation}
where $C$ is a constant which depends on $c$ and $A$.

We still denote by $\eta$ the cut-off function given above. Now we put $\rho(y)=d(y, {\cal S}_{x,1})$.
Now we define as follows: If $\rho(y) \,\ge\, \bar\epsilon /3$, put $\gamma_{\bar\epsilon} (y)\,=\, \chi (y)$ and if $\rho(y)\,<\,\bar\epsilon$, put
\begin{equation}
\label{eq:gamma-1}
\gamma_{\bar\epsilon}(y)\,=\, \chi(y)\,\eta \left ( \log\left (-\log \left(\frac{\rho(y)}{\bar\epsilon} \right)\right)\right).
\end{equation}
Clearly, $\gamma_{\bar\epsilon}$ is smooth. If we choose $\epsilon_0$ and $\delta$ sufficiently small, we have
$\gamma_{\bar\epsilon}(y)\,=\,1$ for any $y$ with $d(y, {\cal S}_x) \,\ge\, \bar\epsilon$, also $\gamma_{\bar\epsilon}$ vanishes in a neighborhood of ${\cal S}_x$.
Furthermore, by using \eqref{eq:prod-1}, the Fubini theorem and our assumption ${\bf A}_1$, we can show
$$\int_{B_{{\bar\epsilon}^{-1}}(o,g_x)} \,|\nabla \gamma_{\bar\epsilon}|^2\, \omega_x^n\,\le\, C' \left (\epsilon_0\,+\, \frac{1}{-\log \delta}\right ),$$
where $C'$ is a constant which may depend on $\bar\epsilon$. Then
the lemma follows under Assumption ${\bf A}_1$ if $\epsilon_0$ and $\delta$ are sufficiently small.

\end{proof}
\vskip 0.1in
Now assuming Lemma \ref{lemm:cut-off-2}, we prove the partial $C^0$-estimate.

First we define $\eta$ to be a cut-off function satisfying:
$$\eta(t)\,= \,1~~{\rm for}~~t\,\le\, 1,~~\eta(t)\,=\,0~~{\rm for}~~ t\,\ge\, 2~~{\rm and}~~|\eta'(t)|\,\le\, 1.$$
Choose $\bar\epsilon$ such that $V(x; \epsilon)$ contains the support of $\gamma_{\bar \epsilon}$ constructed in Lemma \ref{lemm:cut-off-2}
and $\gamma_{\bar \epsilon} \,=\,1$ on $V(x; \delta_0)$, where $\delta_0\, >\,0$ is determined later. Clearly,
we can choose $\bar\epsilon$ as small as we want if $\epsilon$ is sufficiently small.

We define for any $y\,\in \,V(x;  \epsilon)$
$$
\tilde \tau (\phi(y))\,=\,\eta(2 \delta \rho_x(y))\, \eta( 2 \delta \rho_x(y)^{-1} ) \, \gamma_{\bar\epsilon}(y)
\,\tau (\phi(y)).
$$
Clearly, $\tilde \tau$ vanishes outside $\phi(V(x; \epsilon))$, therefore, it extends to a smooth section of $K_{M_\infty}^{-\ell}$ on
$M_\infty$. Furthermore, $\tilde \tau $ satisfies:
\vskip 0.1in
\noindent
(i) $\tilde \tau \,=\, \tau $ on $\phi(V(x; \delta_0))$;

\vskip 0.1in
\noindent
(ii) There is an $\nu\,=\,\nu(\delta,\epsilon)$ such that
$$\int_{M_\infty} ||\bar\partial \tilde \tau||^2\, \omega_\infty^n \,\le \, \nu\, r^{2n-2}.$$
Note that we can make $\nu$ as small as we want so long as $\delta$, $\epsilon$ and $\bar \epsilon$ are sufficiently small.
\vskip 0.1in

Since $(M\backslash D,\omega_i)$ converge to $(M_\infty\backslash S, \omega_\infty)$ and
the Hermitian metrics $H_i$ on $K_M^{-1}$ converge to $H_\infty$ on $M_\infty\backslash S$ in the $C^\infty$-topology.
Therefore, there are diffeomorphisms
$$\tilde\phi_i : M_\infty\backslash {\cal S}\,\mapsto\, M\backslash T_i(D) $$
and smooth isomorphisms
$$F_i:K_{M_\infty}^{-\ell}\,\mapsto\, K_M^{-\ell} $$
over $M\backslash T_i(D)$, where $T_i(D)$ is the set of all points
within distance $\delta_i$ from $D$ with respect to the metric $\omega_i$, where $\delta_i \,>\,0$ and $\lim \delta_i\, =\, 0$,
satisfying:

\vskip 0.1in
\noindent
${\bf C}_1$:  $\tilde\phi_i(M_\infty\backslash T_{\delta_i}({\cal S}))\,\subset \,M\backslash T_\delta (D)$,
where $T_{\delta_i}({\cal S})\,=\,\{x\in M_\infty~|~d_\infty(x, {\cal S})\,\le\, \delta_i\}$;

\vskip 0.1in
\noindent
${\bf C}_2$:  $\pi_i \circ F_i\,=\, \tilde\phi_i\circ \pi_\infty$, where $\pi_i$ and $\pi_\infty$ are corresponding projections;

\vskip 0.1in
\noindent
${\bf C}_3$: $||\tilde\phi_i^*\omega_i - \omega_\infty||_{C^2(M_\infty\backslash T_{\delta_i}({\cal S}))} \,\le\, \delta_i$;

\vskip 0.1in
\noindent
${\bf C}_4$: $||F_i^* H_i - H_\infty||_{C^4(M_\infty\backslash T_i({\cal S}))} \,\le\, \delta_i$.

\vskip 0.1in
We may assume that $i$ is sufficiently large so that $\phi(V(x;\epsilon))\,\subset \,M\backslash T_i({\cal S})$.
Put $\tilde\tau_i\,=\,F_i(\tilde\tau)$, then we deduce from the above
\vskip 0.1in
\noindent
(i') $\tilde \tau_i \,=\, F_i(\tau) $ on $\tilde\phi_i (\phi(V(x; \delta_0)))$;

\vskip 0.1in
\noindent
(ii') For $i$ sufficiently large, we have
$$\int_{M_i} ||\bar\partial \tilde \tau_i||_i^2\, \omega_i^n \,\le \, 2 \nu\, r^{2n-2},$$
where $||\cdot||_i$ denotes the Hermitian norm corresponding to $H_i$.
\vskip 0.1in

By the $L^2$-estimate in Lemma \ref{lemm:4-2}, we get a section $v_i$ of $K_{M_i}^{-\ell }$ such that
$$\bar\partial v_i \,=\, \bar\partial \tilde \tau_i$$
and
$$\int_{M_\infty} ||v_i||_i^2 \,\omega_i^n \,\le \, \frac{1}{\ell} \int_{M_i} ||\bar\partial \tilde \tau_i||_i^2 \,\omega_\infty^n \,\le\, 3 \nu\, r^{2n}  .$$
Here we have used the fact that $\ell$ is the integral part of $r^{-2}$.

Put $\sigma_i \,=\, \tilde \tau_i \,- \,v_i$, it is a holomorphic section of $K_{M_i}^{-\ell}$.
By (i) and Lemma \ref{lemm:4-4}, the $C^4$-norm of $\bar\partial v_i$ on $\tilde\phi_i(\phi(V(x; \delta_0)))$ is bounded from above by
$c \delta $ for a uniform constant $c$. By the standard elliptic estimates, we have
$$\sup _{\tilde\phi(\phi(V(x; 2\delta_0)\cap B_1(o,g_x)))} ||v_i||_i^2 \,\le \, C \,(\delta_0 r)^{-2n}\, \int_{M_i} ||v_i||_i^2\, \omega_i^n
\,\le\, C \,\delta_0 ^{-2n} \,\nu .$$
Here $C$ denotes a uniform constant. For any given $\delta_0$, if $\delta$ and $\epsilon$ are sufficiently small, then we can make $\nu$ such that
$$ 8 C\,\nu\,\le\, \delta_0^{2n}.$$
Then we can deduce from the above estimates
$$||\sigma_i||_i \,\ge\, ||F_i(\tau)||_i\,-\,||v_i||_i\,\ge \,\frac{1}{2}~~~{\rm on}~~\tilde\phi_i(\phi(V(x; \delta_0)\cap B_1(o,g_x))).$$
On the other hand, by applying the derivative estimate in Corollary \ref{coro:3-1} to $\sigma_i$, we get
$$\sup_{M_i} ||\nabla \sigma_i||_i \,\le\, C'
\ell ^{\frac{n+1}{2}} \left (\int_{M_i} ||\sigma_i||_i^2\,\omega_i^n\right )^{\frac{1}{2}} \,\le\, C'\, r^{-1} .$$
Since the distance $d(x, \phi(\delta_0 u)) $ is less than $ 10 \delta_0 r$ for some $u \,\in\,\partial B_1(o,g_x)$, if $i$ is sufficiently large,
we deduce from the above estimates
$$||\sigma_i||_i(x_i)\, \ge\, 1/4 - C'\,\delta_0,$$
hence, if we choose $\delta_0$ such that $C' \delta_0 < 1/8$, then $\rho_{\omega_i,\ell} (x_i) > 1/8$. Theorem \ref{th:main-1}, i.e., the partial $C^0$-estimate
for conic K\"ahler-Einstein metrics, is proved.

As indicated in \cite{tian09} for smooth K\"ahler-Einstein metrics, by the arguments in the proof of the partial $C^0$-estimate, we can prove
the following regularity for $M_\infty$:

\begin{theo}
\lab{th:conj-2}
The Gromov-Hausdorff limit $M_\infty$ is a normal variety embedded in some $\CC P^N$ and $S$ is a subvariety consisting a divisor $D_\infty$ and a subvariety
of complex codimension at least $2$. Moreover, $D_\infty$ is the limit of $D$ under the Gromov-Hausdorff convergence.
\end{theo}

\begin{proof}
For the readers' convenience, we include a proof. Let us recall some well-known facts (cf, \cite{tian09}).
For any $i$ and sufficiently large $\ell$, we can choose
an orthonormal basis $\{\sigma_{i,\ell}\}$ of $H^0(M, K_M^{-\ell})$ with respect to $\omega_i$ and use this to define a Kodaira embedding
$$\psi_{i,\ell}: M \,\mapsto\, \CC P^{N_\ell},~~~{\rm where}~N_\ell + 1 \,=\, \dim H^0(M, K_M^{-\ell}).$$
By using the $L^2$-estimate for $\bar\partial$-operator, we can find an exhaustion of $M_\infty\backslash S$ by open subsets
$V_1\subset V_2\subset \cdots \subset V_\ell \subset \cdots$
such that $\psi_{i,\ell}$ converge to an embedding
$$\psi_{\infty, \ell}: V_\ell \subset M_\infty \,\mapsto\, \CC P^{N_\ell}.$$

By the partial $C^0$-estimate, there is an integer $m > 0$ such that for any $\ell =  m k $, $\psi_{i,\ell}$ converge to an extension of
$\psi_{\infty ,\ell}$ on $M_\infty$ under the convergence of $(M,\omega_i)$ to $(M_\infty, \omega_\infty)$. We still denote this extension by
$$\psi_{\infty, \ell}: \,M_\infty\,\mapsto\, \CC P^{N_{\ell_a}}.$$
By the estimate in Corollary \ref{coro:3-1},
$\psi_{i,\ell}$ are uniformly Lipschtz, so $\psi_{\infty,\ell}$ is a Lipschtz map.

\vskip 0.1in
\noindent
{\bf Claim}: $M_\infty$ is a variety. 
\vskip 0.1in
\noindent
For this, we only need to show that for $k\ge n+1$, $\psi_{\infty,\ell} $ is a homeomorphism
from $M_\infty$ onto its image which is also the limit of complex submanifolds $\psi_{i,\ell}(M)\subset \CC P^{N_\ell}$.

By the same arguments as those in proving the partial $C^0$-estimate, for any $r > 0$, there are $k(r)$ and $s(k)$ such that if $k \ge k(r)$,
then for any $x, y\in M$ such that $d_i(x,y) \,\ge\, r$, where $d_i(\cdot,\cdot)$ denotes the distance of the metric $\omega_i$,
there is a holomorphic section
$\varsigma_i \in H^0(M, K_M^{-\ell})$, where $\ell = m k$,
satisfying:
\begin{equation}
\label{eq:separation}
\int_M \,||\varsigma_i||_i^2 \omega_i^n\,=\,1~~~{\rm and}~~~| ||\varsigma_i||_i(x) \,-\,||\varsigma_i||_i(y) |\,\ge\, s(k).
\end{equation}

The above claim follows from this and the effective finite generation of the anti-canonical ring of $M$ as shown in the thesis of Chi Li \cite{chili}.
\footnote{As I advocated in many occasions before (cf. \cite{tian09}), the partial $C^0$-estimate corresponds to an effective version of
the finite generation of the anti-canonical ring. Chi Li showed precisely in \cite{chili} how this works.}
For the orthonormal basis $\{\sigma_{i,a}\}_{0\le a\le N_m}$ of $H^0(M, K_M^{-m})$ with respect to $\omega_i$, by the partial $C^0$-estimate
and Corollary \ref{coro:3-1}, we have
\begin{equation}
\label{eq:pC0-1}
c(m)\,\le\,\sum_{a=0}^{N_m} \,||\sigma_{i,a}||_i^2 \,\le\, c(m)^{-1},
\end{equation}
where $c(m)$ is a uniform constant independent of $i$.

\begin{lemm}
\label{lemm:chili}
For any $l \,\ge\,1$ and $\varsigma\,\in\, H^0(M, K_M^{-(n + 1+l) m})$,
there are $h_0,\cdots,h_{N_m}$ in $H^0(M, K_M^{-(n+l)m})$ satisfying:
\begin{equation}
\label{eq:chili}
\varsigma \,=\, \sum_{a=0}^{N_m} \,h_a \,\sigma_{i,a}~~~{\rm and}~~~\int_M \,||h_a||_i^2 \,\omega_i^n\,\le\, C(m,l)\,\int_M ||\varsigma||_i^2\,\omega^n_i,
\end{equation}
where $C(m,l)$ is a constant depending only on $c(m)$, $l$ and $n$.

\end{lemm}
This can be proved by using the Skoda-Siu estimate, now a standard technique (cf. \cite{chili}, Proposition 7).

Note that for any $x\in M_\infty$ and $k\ge 1$, we have
\begin{equation}
\label{eq:obser-1}
\psi_{\infty, m k }^{-1}( \psi_{\infty, m k }(x))\,\subseteq \, \psi_{\infty, m }^{-1}( \psi_{\infty, m }(x)).
\end{equation}
Using this and Lemma \ref{lemm:chili} with $i\to \infty$, we get
$$\psi_{\infty, m (n+1 + l) }^{-1} ( \psi_{\infty, m (n+1+l)} (x))\,\supseteq\,\psi_{\infty, m (n+1) }^{-1}(\psi_{\infty, m (n+1) }(x)).$$
It follows from \eqref{eq:separation} that for any $x\not= y \in M_\infty$,
$$\psi_{\infty, m (n+1 + l) }(x)\,\not=\,\psi_{\infty, m (n+1 + l) }(y)$$
if $l$ is sufficiently large. Therefore, we can get
$$\psi_{\infty, m (n+1 ) }(x)\,\not=\,\psi_{\infty, m (n+1 ) }(y).$$
This implies that $\psi_{\infty, m (n+1)}$ is a homeomorphism, so $M_\infty$ is a variety.

There is another way of proving that $\psi_{\infty, m k}$ is a homeomorphism for $k$ sufficiently large. By \eqref{eq:obser-1},
the composition $\psi_{\infty, m }\cdot \psi_{\infty, mk}^{-1}$ is a well-defined map from the variety $Y_{mk}$ onto $Y_m$, where
$$Y_{mk}\,=\,\lim_{i\to\infty} \psi_{i,mk}(M) \subset \CC P^{N_{mk}},~~~Y_m \,=\, \lim_{i\to \infty} \psi_{i,m}(M) \subset \CC P^{N_{m}}.$$
Moreover, this map is also the limit of holomorphic maps $\psi_{i,m}\cdot\psi_{i,mk}^{-1}$, so
it is a holomorphic map. Since $\psi_{\infty, m}$ restricted to $V_m$ is an embedding for $m$ sufficiently large, we know
that $\psi_{\infty, m k }(\psi_{\infty, m}^{-1}(z))$ is either a point or a connected subvariety in the complex limit space
$Y_{mk}$. The second case can be ruled out by using the fact that there is a bounded function $u$ such that
$$\frac{1}{m k} \omega_{FS}|_{Y_{mk}}\,=\, \frac{1}{m} (\psi_{\infty, m }\cdot \psi_{\infty, m k}^{-1})^* (\omega_{FS}|_{Y_{m}})\,+\,\sqrt{-1}\partial\bar\partial u,$$
where $\omega_{FS}$ always denotes the Fubini-Study metric.

Next we prove that $M_\infty$ is normal. This means that $M_\infty\backslash S$
is locally connected. If $\beta_\infty= 1$, it is trivially true since the singular set of $M_\infty$ is of complex codimension at least $2$.
So we may assume $\beta_\infty < 1$. There are several approaches. One can use
a local version of the Cheeger-Gromoll splitting theorem (cf. \cite{anderson}).
One can also generalize the arguments I had in \cite{tian89} or use
the Cheeger-Colding theory.

Before we prove the normality of $M_\infty$, we make a remark: By Corollary \ref{coro:3-1} and the partial $C^0$-estimate,
$\log \,\rho_{\omega_i,m}$ converge to a uniformly continuous function $\log \,\rho'_{\infty,m}$ on $M_\infty$. This implies that
$\omega_\infty$ is the curvature of a continuous Hermitian metric on $K_{M_\infty}^{-1}$, so $||\cdot||_\infty$ is a continuous Hermitian metric on $M_\infty$
even when $\beta_\infty < 1$. Therefore, the defining section $\sigma_i$ of $D$ normalized by $\omega_i$ converge to a holomorphic section $\sigma_\infty$
of $K_{M_\infty}^{-\lambda}$. Clearly, the singular set ${\cal S}$ of $(M_\infty, \omega_\infty)$ is the divisor $D_\infty$ defined by $\sigma_\infty$ possibly
plus a closed subset
${\cal S}_{2n-4}$ of
complex codimension at least $2$.

Therefore, if $M_\infty$ is not normal, then $M_\infty\backslash D_\infty$ is not locally connected near a point, say $x$, in $D_\infty$.
Since $x\in {\cal S}\backslash \bar {\cal S}_{2n-4}$, there is a tangent cone ${\cal C}_x$ of $M_\infty$ at $x$ of the form $\CC^{n-1}\times {\cal C}_x'$, where ${\cal C}_x'$ is a 2-dimensional flat cone of angle $2\pi\bar\beta$, where $(1-\bar\beta) = k (1-\beta_\infty)$. 
However, ${\cal C}_x\backslash {\cal S}_x$ is connected, so $M_\infty\backslash D_\infty$ is connected near $x$, a contradiction.
Therefore, $M_\infty$ must be normal.

Note that the normality also follows from a result of Colding-Naber who proved the convexity of $M_\infty\backslash {\cal S}$.

\end{proof}

Of course, one can further analyze the finer asymptotic structure of $\omega_\infty$ along $D_\infty$. By the partial $C^0$-estimate and Corollary \ref{coro:3-1},
we have
$$\omega_\infty\,\ge \, c \,\psi_{\infty,\ell}^*(\omega_{FS}),$$
where $\ell = m k$ and $c$ is some positive constant. Using this, when $\beta_\infty \,<\, 1$, one can show that $\omega_\infty$ is a conic K\"ahler-Einstein metric
with conic angle $2\pi\bar\beta$ along $D_\infty$ in a weaker sense, where $(1-\bar\beta) = k (1-\beta_\infty)$\footnote{The integer $k$ may vary on different 
connected components of $D_\infty$.}. It is an interesting problem to examine the precise behavior of $\omega_\infty$ along $D_\infty$.

The following theorem may be useful in the future.
\begin{theo}
\label{th:4-2}
For each $\ell > 0$, let $\{\sigma_{i,\alpha}\}$ be an orthonormal basis of $H^0(M,K_M^{-\ell})$. Then by taking a subsequence if necessary,
$\{\sigma_{i,\alpha}\}$ converge to an orthonormal basis $\{\sigma_{\infty,\alpha}\}$ of $H^0(M_\infty, K_{M_\infty}^{-\ell})$. In particular, it implies
that $H^0(M_\infty, K_{M_\infty}^{-\ell})$ is of finite dimension and $\rho_{\omega_i,\ell}$ converge to $\rho_{\omega_\infty,\ell}$ as $i$ tends to $\infty$.

\end{theo}

\begin{proof}
The arguments appeared before (cf. \cite{tian09}) and
are based on the $L^2$-estimate for the $\bar\partial$-operator. In view of Lemma \ref{lemm:conv},
it suffices to prove that
any $\tau$ in $H^0(M_\infty, K_{M_\infty}^{-\ell})$ with its $L^2$-norm being one
is a limit of a sequence $\tau_i \in H^0(M,K_M^{-\ell})$.
We will adopt the notations in establishing of the partial $C^0$-estimate, particularly, ${\bf C}_1$-${\bf C}_4$.

The following lemma is an analogue of Lemma \ref{lemm:cut-off-2}. It is easy to prove by
using Theorem \ref{th:conj-2}.
\begin{lemm}
\label{lemm:cut-off}
For any $\epsilon \,>\,0$, there is a smooth function $\gamma_\epsilon$ on $M_\infty$ satisfying:

\vskip 0.1in
\noindent
(1) $\gamma_\epsilon(x)\,=\,1$ for any $x$ with $d_\infty(x,{\cal S})\,\ge\,\epsilon$;

\vskip 0.1in
\noindent
(2) $0\,\le\,\gamma_\epsilon \,\le\,1$ and $\gamma_\epsilon(x)\,=\,0$ in an neighborhood of ${\cal S}$;

\vskip 0.1in
\noindent
(3) $|\nabla \gamma_\epsilon|\,\le\, C$ for some constant $C\,=\,C(\epsilon)$ and
$$\int_{M_\infty} \,|\nabla \gamma_\epsilon|^2\, \omega_\infty^n\,\le\,\epsilon.$$
\end{lemm}

For each $i$ and $\epsilon \in (0,1) $, define
$$ \xi_\epsilon (x) \,=\,F^*_i( \gamma_\epsilon \,\tau ) (x) ,$$
Then $\xi_\epsilon$ is a smooth section of $K_{M}^{-\ell}$ satisfying:

\vskip 0.1in
\noindent
(1) $\xi_\epsilon (x)= 0 $ in an neighborhood of $S$;

\vskip 0.1in
\noindent
(2) put $\zeta_\epsilon = \bar \partial \xi_\epsilon$, then
\begin{equation}
\label{eq:4-3}
\int_M ||\zeta_\epsilon ||^2 \,\omega_i^n \,\le\,  2\,\int_{M} || \bar \partial( F^*_i\tau)||^2 \,\omega_i^n \,+\, C\,\epsilon\,\sup_{M_\infty}||\tau||_\infty^2 ,
\end{equation}
where $C$ is a uniform constant.

Let $\delta_i$ be given in ${\bf C}_1$-${\bf C}_4$. Then there are $\epsilon_i$ with $\lim \epsilon_i=0$ such that
$\gamma_{\epsilon_i}$ in the above lemma vanishes in an neighborhood of the closure of $T_{\delta_i}({\cal S})$. Put $\xi_i\,=\,\xi_{\epsilon_i}$
and $\zeta_i\,=\,\zeta_{\epsilon_i}$, then it follows from
\eqref{eq:4-3} that
\begin{equation}
\label{eq:4-4}
\lim _{i\to \infty } \int_M ||\zeta_i||^2 \,\omega_i^n \,=\, 0.
\end{equation}
Applying Lemma \ref{lemm:4-2} to $\zeta_i$, we get $\vartheta_i$ such that $\bar \partial \vartheta_i = \zeta_i$ and
$$\int _M ||\vartheta_i||_i^2 \,\omega_i^n \,\le \, \frac{1}{\ell + \mu}\,\int_M ||\zeta_i||_i^2\,\omega_i^n\, \to \,0.$$
On the other hand, $\zeta_i = \bar\partial \xi_i$. By the construction of $\xi_i$, we can easily show that $\xi_i$ converge to $\tau$ in the $C^\infty$-topology
outside ${\cal S}$ and
$$\lim_{i\to\infty} \int _M ||\xi_i||_i^2 \,\omega_i^n \,=\, \int_{M_\infty} ||\tau||_\infty^2\,\omega_\infty^n \,>\,0.$$
Then $\tau_i\,=\,\xi_i-\vartheta_i$ defines a holomorphic section of $K_M^{-\ell}$ which converges to $\tau$ in the $L^2$-topology. By the standard elliptic estimates,
we can easily show that $\tau_i$ converge to $\tau$ in the $C^\infty$-topology outside ${\cal S}$. This proves Theorem \ref{th:4-2}.
\end{proof}

\section{Proving Theorem \ref{th:main-2}}

In this section, we complete the proof of Theorem \ref{th:main-2}, i.e., if a Fano manifolds $M$ is K-stable, then it admits a K\"ahler-Einstein metric.
In fact, as I pointed out in describing my program on the existence of K\"ahler-Einstein metrics, the reduction of Theorem \ref{th:main-2} from the partial $C^0$-estimate had been known to me for long. \footnote{Our program was originally proposed for the Aubin continuity method, but it works for the
new Donaldson-Li-Sun continuity method in an identical way.}

As explained in the introduction, in order to prove Theorem \ref{th:main-2}, we only need to establish the $C^0$-estimate for the solutions
of the complex Monge-Ampere equations for $\beta > 1-\lambda^{-1}\,$:
\begin{equation}
\label{eq:5-1}
(\omega_\beta \,+\, \sqrt{-1} \partial \bar\partial \varphi )^n \,=\, e^{h_\beta - \mu \varphi} \omega^n_\beta,
\end{equation}
where $\omega_\beta$ is a suitable family of conic K\"ahler metrics with $[\omega_\beta] = 2\pi c_1(M)$ and cone angle $2\pi \beta$ along $D$ and
$h_\beta$ is determined
by
$${\rm Ric}(\omega_\beta)\,=\,\mu \omega \,+\,2\pi (1-\beta) [D] \,+ \,\sqrt{-1} \partial\bar\partial h_\beta~~{\rm and}~~\int_M( e^{h_\beta} - 1) \omega^n_{\beta}\,=\,=0.$$
By the discussed in the introduction, we know that there is a non-empty and maximal interval $E\,=\, (1-\lambda^{-1}, \bar\beta)$ for some $\bar\beta \in  (1-\lambda^{-1}, 1 )$
or $(1-\lambda^{-1},1]$ such that \eqref{eq:5-1} has a solution $\varphi_\beta$ for any $\beta\in E$. Actually, such a solution $\varphi_\beta$ is unique, so $\{\varphi_\beta\}$ is a continuous family on $M$ and smooth outside $D$.\footnote{In fact, one can use prove this continuity and smoothness directly by using the Inverse Function Theorem as we argued for the openness of $E$.}
If $1\in E$, we already have Theorem \ref{th:main-2} and nothing more needs to be done. Hence, we may assume that $E=(1-\lambda^{-1}, \bar\beta)$ for some
$\bar\beta \,<\, 1$, we will derive a contradiction. By our assumption and the results in \cite{jeffresmazzeorubinstein}, $||\varphi_\beta||_{C^0}$ diverge to $\infty$ as $\beta$ tends to $\bar\beta$. We will show that it contradicts to the K-stability of $M$. Now let us recall the definition of the K-stability.
I will use the original one from \cite{tian97} which is directly related to our program of establishing the existence of K\"ahler-Einstein metrics through the continuity method.

First we recall the definition of the Futaki invariant \cite{futaki83}:
Let $M_0$ be any Fano manifold and $\omega$ be a K\"ahler metric with $c_1(M)$ as its K\"ahler class,
for any holomorphic vector field $X$ on $M_0$, Futaki defined
\begin{equation}
\lab{eq:5-2}
f_{M_0}(X)\,=\, \int_M X(h_\omega) \,\omega^n,
\end{equation}
where ${\rm Ric}(\omega) \,-\, \omega \,=\, \sqrt{-1} \partial\bar\partial h_\omega$.
Futaki proved in \cite{futaki83} that $f_M(X)$ is independent of the choice of $\omega$, so it is a holomorphic invariant.
In \cite{dingtian92}, the Futaki invariant was extended to normal Fano varieties. The extension is based on the following reformulation:
\begin{equation}
\lab{eq:5-3}
f_{M_0}(X) \,=\, - \,n \int_M \theta_X  \left ({\rm Ric}(\omega) - \omega \right ) \wedge \omega^{n-1},
\end{equation}
where $i_X\omega \,=\, \sqrt{-1}\, \bar\partial \theta_X$.

Now let $M$ be a Fano manifold $M$. By the Kodaira embedding theorem, for $\ell$
sufficiently large, any basis of $H^0(M, K_M^{-\ell})$ gives an embedding
$$\phi_\ell: M \mapsto \CC P^{N},$$
where $N\,=\,\dim_\CC H^0(M,K_M^{-\ell})-1$. Any other basis gives an embedding of the form
$\sigma\circ \phi_\ell$, where $\sigma\in G=\SL(N+1, \CC)$.

For any algebraic subgroup $G_0\,=\,\{\sigma(t)\}_{t\in \CC ^*}$ of
$\SL(N+1, \CC)$, there is a unique limiting cycle
$$M_0\,=\,\lim_{t\to 0} \sigma(t)(M)\subset \CC P^N.$$
Let $X$ be the holomorphic vector field whose real part generates the action by $\sigma(e^{-s})$.
By \cite{dingtian92}, if $M_0$ is normal, there is a generalized Futaki invariant $f_{M_0}(X)$ defined by \eqref{eq:5-3}.

Now we can introduce the K-stability from \cite{tian97}.

\begin{defi}
\lab{defi:5-1}  We say that $M$ is K-stable with respect to
$K_M^{-\ell}$ if $f_{M_0}(X) \ge 0$ for any $G_0\subset \SL(N+1)$ with a normal $M_0$
and the equality holds if and only if $M_0$ is biholomorphic to $M$. We say that $M$ is K-stable if it is K-stable for all sufficiently large $\ell$.
\end{defi}

There are other formulations of the K-stability by S. Donaldson in \cite{donaldson02} and S. Paul in \cite{paul08}.

It was proved in \cite{tian97}
\begin{theo}
\label{th:5-1}
Let $M$ be a Fano manifold without non-trivial holomorphic vector fields and which admits a K\"ahler-Einstein metric. Then
$M$ is K-stable.
\end{theo}

Now we return to our Fano manifold $M$ in Theorem \ref{th:main-2} and those solutions $\varphi_\beta$ ($\beta\in E$) as above.
In order to get a contradiction, we need to produce only a normal Fano variety $M_0$ as in Definition \ref{defi:5-1} and with non-positive Futaki invariant.


Let $\{\beta_i\}$ be a sequence with $\lim \beta_i \,=\, \bar\beta$. Write $\varphi_i\,=\,\varphi_{\beta_i}$. If $\sup_M \varphi_i$
is uniformly bounded, by the Harnack-type estimate in Theorem in \cite{jeffresmazzeorubinstein},
the $C^0$-norm of $\varphi_i$ is uniformly bounded. So, by \cite{jeffresmazzeorubinstein} again, $\varphi_i$ converge to a solution of \eqref{eq:5-1} for $\beta\,=\,\bar\beta$. A contradiction! Therefore, we have
$$ \lim_{i\to \infty} \sup_M \varphi_{i}\,=\, \infty.$$
We will fix such a sequence $\{\beta_i\}$ and write
$$\omega_i \,=\, \omega \,+\, \sqrt{-1}\,\partial\bar\partial \varphi_{i}.$$
Then $\omega_i$ is a conic K\"ahler-Einstein metric on $M$ with cone angle $2\pi \beta_i$ along $D$. By taking a subsequence of necessary,
we may assume that $(M,D,\omega_i)$ converge to $(M_\infty,D_\infty,\omega_\infty)$ in the Gromov-Hausdorff topology.
By Theorem \ref{th:conj-2}, $M_\infty$ is a normal subvariety in some projective space $\CC P^N$ and $\omega_\infty$ is a smooth K\"ahler-Einstein metric
outside a divisor $D_\infty$ and the singular set $\bar {\cal S}$ of $M_\infty$.\footnote{We have seen in last section that $\omega_\infty$ has
locally continuous potentials.} 
We will identify $M_\infty$ with its image in $\CC P^N$ by an embedding defined by a basis $\{S_{\infty,\alpha}\}$ of $H^0(M_\infty, K_{M_\infty}^{-\ell})$,
in fact, such a basis
$\{S_{\infty,\alpha}\}$ is orthonormal with respect to the inner product on $H^0(M_\infty, K_{M_\infty}^{-\ell})$ by $\omega_\infty$.

Similarly, we embed $M$ by orthonormal bases of $H^0(M,K_M^{-\ell})$ with respect to $\omega_i$. All these embeddings differ by transformations in $G$.
On the other hand, by taking a subsequence if necessary, we may assume that those orthonormal bases of $H^0(M,K_M^{-\ell})$ converge to the orthonormal basis 
$\{S_{\infty,\alpha}\}$ of $H^0(M_\infty,K_{M_\infty}^{-\ell})$ under
the convergence of $(M,D,\omega_i)$ to $(M_\infty, D_\infty, \omega_\infty)$. It implies that $(M_\infty, D_\infty)$ lies in the closure of the orbit of $(M, D)$ under the group action
of $G\,=\,\SL(N+1,\CC)$ on $\CC P^N$. Then one can deduce from some general facts in algebraic geometry that the stabilizer $G_\infty$ of $M_\infty$ in $G$ contains a holomorphic subgroup.\footnote{For the Aubin continuity, one can show by geometric analytic arguments that $M_\infty$ admits a $\CC^*$-action.
One should be able to extend this method to the continuity method proposed by Donaldson et al.} We need to prove that it contains a $\CC^*$-subgroup. Then, using the K\"ahler-Einstein metric $\omega_\infty$, one can show that the generalized Futaki invariant is not positive. This contradicts to the K-stability.

\begin{lemm}
\label{lemm:5-1}
The Lie algebra $\eta_\infty$ of $G_\infty$ is reductive.

\end{lemm}
\begin{proof}
The arguments are standard. Let $X\in \eta_\infty$, i.e., a holomorphic vector field on $\CC P^N$ which is tangent to $M_\infty$, then there is a smooth
function $\theta$ such that $i_X \omega _{FS}\,=\,\ell \sqrt{-1}\,\bar\partial \theta$. We have
$$\ell\,\omega_\infty\,=\,\omega_{FS}|_{M_\infty}\,+\,\sqrt{-1}\, \partial\bar\partial \rho_{\omega_\infty,\ell}.$$
It follows
$$i_X\omega_\infty\,=\,\sqrt{-1}\,\bar\partial \theta_\infty,~~~{\rm where}~\theta_\infty\,=\, \theta \,+\,\frac{1}{\ell}\,X(\rho_{\omega_\infty,\ell}).$$
It is a fact that $X$ generates a $\CC^*$-action if and only if it is a complexication of a Killing field. Therefore, if we normalize $X$ by multiplication by a complex number such that $\sup_{M_\infty} \theta_\infty \,=\, 1$, we want to show that the imaginary part of $X$ is Killing. The standard
computations show that if $\theta_\infty$ is normalized by
$$ \int_{M_\infty} \theta_\infty \,\omega_\infty^n\,=\,0,$$
then
$$ \Delta_\infty \theta_\infty \,+\,\mu_\infty\,\theta_\infty\,=\,0~~~{\rm on}~M_\infty\backslash D_\infty\cup \bar{\cal S},$$
where $\Delta_\infty$ denotes the Laplacian of $\omega_\infty$ and $\mu_\infty = 1 - (1-\bar \beta)\lambda$.
On the other hand, by using our estimates on $\rho_{\omega_\infty,\ell}$ and the Bochner identity, we can show that $\theta_\infty$ is Lipschtz continuous, thus it extends to an eigenfunction of $\Delta_\infty$, so do its real and imaginary parts. It follows from the standard arguments that the imaginary part of
$\theta_\infty$ induces a Killing field. Then the lemma is proved.
\end{proof}

As observed in \cite{donaldson11} and \cite{li11}), by using the same arguments as in \cite{futaki83}, one can define the Futaki invariant $f_{M_\infty,(1-\beta)D_\infty}(X)$, also referred as the log-Futaki invariant, for conic
K\"ahler metrics on $M_\infty$ with cone angle $2\pi \beta$ along $D_\infty$ ($\beta\in (0,1)$). Furthermore,
if there is a conic K\"ahler-Einstein metric with angle $2\pi\beta$ along $D_\infty$, the log-Futaki $f_{M_\infty, (1-\beta) D_\infty}$ vanishes. In our case, though
$\omega_\infty$ may not be smooth along $D_\infty$ even in the conic sense, using the Lipschtz continuity of $\theta_\infty$, one can still prove
the vanishing of $f_{M_\infty, (1-\beta) D_\infty}(X)$ by the same arguments as in the smooth case.
Then the Futaki invariant $f_{M_\infty}(X)\,\le \,0$.
This can be derived by using the formula (cf. \cite{li11}, \cite{sun11}):\footnote{Chi Li pointing out that this formula first appeared in
\cite{donaldson11}. I thank him for this as well as some other inputs on log-Futaki invariants.}
$$0\,=\,f_{M_\infty, (1-\beta) D_\infty}(X)\,=\,f_{M_\infty}(X) \,+\,(1-\beta)\,\int_{D_\infty} \theta_\infty\, d{\cal H}^{2n-2},$$
where $d{\cal H}^{2n-2}$ denotes the (2n-2)-dimensional Hausdorff measure on $D_\infty$ induced by $\omega_\infty$. To see this, we
first observe that $f_{M_\infty, (1-\beta_1) D_\infty}(X) > 0 $ for some $\beta_1 \in (1-\lambda^{-1},\beta)$, e.g., if it is sufficiently close to
$1- \lambda^{-1}$ because there is a corresponding conic
K\"ahler-Einstein metric with angle $2\pi \beta_1$, on the other hand, because of the linearity, we have
$$(\beta-\beta_1)\,f_{M_\infty}(X)\,=\, (1-\beta_1) \,f_{M_\infty, (1-\beta) D_\infty}(X)\,-\,(1-\beta)\,f_{M_\infty, (1-\beta_1) D_\infty}(X),$$
hence, $f_{M_\infty}(X)\le 0$.

On the other hand, by our assumption that $M$ is K-stable, since $M_\infty$ is not biholomorphic to $M$,
$$f_{M_\infty}(X) > 0.$$
This is a contradiction! Therefore, $\varphi_\beta$ are uniformly bounded and consequently, $\bar\beta\in E$, so $E$ is closed and Theorem \ref{th:main-2} is proved.

There is another way of finishing the proof of Theorem \ref{th:main-2} by using the CM-stability introduced in \cite{tian97}.
The CM-stability can be regarded as a geometric invariant theoretic version of the K-stability. It follows from
\cite{paultian} and \cite{paul08} that the CM-stability is equivalent to the K-stability.
In the following, we outline this alternative proof of Theorem \ref{th:main-2}.

Let us recall the CM-stability. We fix an embedding $M\subset \CC P^N$ by $K_M^{-\ell}$ as above.
Let $\pi: {\cal X}\mapsto Z$ be the universal family of $n$-dimensional normal varieties \footnote{Normality is not needed, but we assume this for simplicity.
Also by \cite{lixu11}, this assumption does not put any constraints on our results.} in $\CC P^N$ with the same Hilbert polynomial as that of $M$.
Clearly, $G\,=\,\SL(N+1)$ acts both ${\cal X}$ and $Z$ such that $\pi$ is equivariant.

Consider the virtual bundle
$${\cal E} \,= \, (n+1) ({\cal K} - {\cal K}^{-1}) ({\cal L} - {\cal L}^{-1})^n - n ({\cal L} - {\cal L}^{-1})^{n+1}, $$
where ${\cal K} \,=\, K_{\cal X} \otimes K^{-1}_{Z}$ is
the relative canonical bundle and ${\cal L}$ is the pull-buck of the hyperplane line bundle on $\CC P^N$.

Let $L$ be the determinant
line bundle $\det ( {\cal E}, \pi )$.
Clearly, $G$ acts naturally on the total space of $L$.

\begin{defi}
\label{defi:5-2}
Let $z=\pi(M)$ and $\tilde z$ be a non-zero lifting of $z$ in the total space of $L$.
We call $M$ CM-stable with respect to $K_M^{-\ell}$ if the orbit $G\cdot\tilde z$ in the total space of $L$
is closed and the stabilizer $G_z$ of $z$ is finite. We call $M$ CM-semistable if $0$ is not in the closure of $G\cdot \tilde z$.
We call $M$ CM-stable if it does with respect to all sufficiently large $\ell$.
\end{defi}

Now we fix $M$, $M_\infty$, $\ell$ as above. Given any $\sigma\in G$, there is an induced K\"ahler potential $\varphi_\sigma$
$$\frac{1}{\ell}\,\sigma^*\omega_{FS}\,=\,\omega_0\,+\, \sqrt{-1}\,\partial\bar\partial \varphi_\sigma.$$
Define a functional on the orbit $G\cdot z$:
$$F_\ell (\sigma) \,=\,{\bf F}_{\omega_0} (\varphi_\sigma).$$
Then we have the following (\cite{tian97}, Theorem 8.10)
\begin{theo}
\label{th:5-3}
The functional $F_\ell$ is proper on $G\cdot z \subset Z$ if and only if $M$ is CM-stable with respect to
$K_M^{-\ell}$.
\end{theo}

By our discussions in Section 3, we can show that ${\bf F}_{\omega_0,\mu}$ restricted to $G\cdot z$ is proper for any $\mu \in (0,1]$. Combining this properness
with the partial $C^0$-estimate, we can bound the $C^0$-norm of $\varphi_\beta$ in a uniform way. Then it follows from \cite{jeffresmazzeorubinstein} that $E$ is closed.
Therefore, we have proved
\begin{theo}
\label{th: CM->KE}
Let $M$ be a Fano manifold without non-trivial holomorphc fields, then $M$ admits a K\"ahler-Einstein metric if and only if $M$ is CM-stable.
\end{theo}
In view of \cite{paultian} and \cite{paul08}, particularly Theorem D in \cite{paul08},
this implies Theorem \ref{th:main-2}.

\section{Appendix: The proof of Lemma \ref{lemm:cut-off-2}}

In this appendix, we complete the proof of Lemma \ref{lemm:cut-off-2}. We will adopt the notations in Section 5, 
particularly, in the proof of those special cases of
Lemma \ref{lemm:cut-off-2}. The arguments of our proof are based on known techniques, though tedious.
Note that if $\beta_\infty = 1$, then there is nothing to be proved since the singular set ${\cal S}_x$ is of complex dimension
at least $2$. So we may assume that $\beta_\infty < 1$. In this case, ${\cal S}_x$ has a decomposition into ${\cal S}_x^0$ and $\bar{\cal S}_x$ as before,
and for any $y \in {\cal S}^0_x$, there is a tangent cone of ${\cal C}_x$ at $y$
of the form $\CC^{n-1}\times {\cal C}_y'$ for which Lemma \ref{lemm:cut-off-2} has been proved.

Fix any $y\in {\cal S}_x^0\subset {\cal C}_x$, we have a tangent cone of the form $\CC^{n-1}\times {\cal C}_y'$ at $y$, where ${\cal C}_y'$ denotes the standard 2-dimensional
cone with angle $2\pi \bar\beta$, where $\bar\beta=\mu_a$ is given as in Lemma \ref{lemm:4-3} and satisfies
$(1-\bar\beta) = k (1-\beta_\infty)$ for some integer $k$.

There are $x_i \in M$ and $r_i> 0$ such that $(M, r_i^{-2} \omega_i, x_i)$ converge to the cone $({\cal C}_x, g_x, o)$
in the Gromov-Hausdorff topology and smooth topology outside the singular set ${\cal S}_x$, in particular, there are diffeomorphisms
$$\tilde\phi_i : V(x;\delta_i) \,\mapsto\, M\backslash T_{\delta_i}(D) ,$$
where $T_{\delta_i} (D)$ is the set of all points
within distance $\delta_i$ from $D$ with respect to the metric $\omega_i$ and $\lim \delta_i\, =\, 0$,
satisfying:

$$||r_i^{-2}\,\tilde\phi_i^*\omega_i - \omega_x||_{C^2(V(x;\delta_i))} \,\le\, \delta_i.$$
Furthermore, we may assume
$$B_{\frac{r_i}{2\delta_i}} (x_i,\omega_i) \backslash T_{\delta_i}(D)\,\subset \phi_i(V(x;\delta_i)).$$
Without loss of generality, we may assume that $\ell_i = r_i^{-2}$ are integers and Lemma \ref{lemm:4-4} holds for such $\ell_i$'s.

Note that there is a tangent cone of the form $\CC^{n-1}\times {\cal C}_y'$ with the standard cone metric $g_{\bar\beta}$
in the proof of Lemma \ref{lemm:cut-off-2}. The singular set of this tangent cone is $\CC^{n-1}\times\{0\}$.
Therefore, there are integers $k_j = s_j^{-2}$ such that $({\cal C}_x, k_j g_x, y)$ converge to $(\CC^{n-1}\times {\cal C}_y', g_{\bar\beta}, o)$ in the Gromov-Hausdorff topology and smooth topology outside the singular set. This implies that there are diffeomorphisms
$$\vartheta_j : U_{j} \,\mapsto\, {\cal C}_x\backslash {\cal S}_x $$
satisfying:
$$||s_j^{-2} \vartheta_j^*\omega_x - \omega_{\bar\beta}||_{C^2(U_j)} \,\le\, \frac{1}{j},$$
where
$$U_j \,=\,\{ (z',z_n) \in \CC^{n-1}\times {\cal C}_y'\,|\,\,|z'|\,<\, 9, ~~~\frac{1}{j}\,<\,|z_n|^{\bar\beta}\,<\, 9\,\}.$$
We may also have
$$B_{9 s_j} (y,g_x) \backslash T_{\frac{2}{j}}({\cal S}_x)\,\subset \vartheta_j (U_j).$$

Combining these, we see that for any $\epsilon > 0$, there are $j_\epsilon$ and $i_\epsilon$ such that for any $j \ge j_\epsilon$ and $i\ge i_\epsilon$,
the compositions
$$\tilde\phi_i\cdot \vartheta_j: U_j \mapsto M\backslash D$$
satisfying:
$$B_{(9-\epsilon) s_j r_i}(x_i, \omega_i)\backslash T_{\delta_i} (D)\,\subset\,
\tilde\phi_i(\vartheta_j(U_j))\,\subset \, B_{13 s_j r_i}(x_i, \omega_i)$$
and
$$||k_j\ell_i \,\vartheta_j^*\tilde\phi_i^*\omega_i - \omega_{\bar\beta}||_{C^2(U_j)} \,\le\, \epsilon$$
Furthermore, by using the above arguments in establishing the partial $C^0$-estimate, given any finitely many holomorphic functions $f_b$ ($b= 0, 1,\cdots,m$) with
$$\int_{\CC^{n-1}\times {\cal C}_y'} \,|f_b |^2\,e^{-\frac{|z'|^2 + |z_n|^{2\bar\beta} }{2}}\,\omega_{\bar\beta}^n\,=\,1,$$
where $\omega_{\bar\beta} $ is the K\"ahler form of $g_{\bar\beta}$, we can construct holomorphic sections $S_{i,j}^b$ of $K^{- k_j\ell_i}_M$ over $M$ such that
$$\sup _{U_j}\,|(\psi_{i,j})^* (S_{i,j}^b) \,-\, f_b | \,\le\, \frac{\epsilon}{2},$$
where $\psi_{i,j}$ is the isomorphism constructed by Lemma \ref{lemm:4-3} over $U_j$. By Corollary \ref{coro:3-1}, for some uniform constant $C$, we have
$$||\nabla S_{i,j}^b||_{i}\,\le\, C.$$
Now we take $f_0$ to be a positive constant function, then $S^0_{i,j}$ is almost a positive constant on $ \tilde\phi_i(\vartheta_j(U_j))$ which contains
$B_{8 s_j r_i}(x_i, \omega_i)$.
Then by rechoosing $j_\epsilon$ and $i_\epsilon$ if necessary, we can deduce from the properties of $S^b_{i,j}$:

\vskip 0.1in
\noindent
{\bf 1}. $B_{8 s_j r_i}(x_i, \omega_i)$ is contained in some $\CC^{N'}$, where $N'$ may depend on $i,j$;

\vskip 0.1in
\noindent
{\bf 2}. There is a holomorphic map $ F_{i,j}^m: \tilde\phi_i(\vartheta_j(U_j))\,\mapsto\, \CC^m$, where
$$ F_{i,j}^m\,=\, \left(\frac{S_{i,j}^1(x)}{S_{i,j}^0(x)},\cdots, \frac{S_{i,j}^m(x)}{S_{i,j}^0(x)}\right)$$
satisfying:
$$\left |F_{i,j}^m(\tilde\phi_i(\vartheta_j(z))) \,-\, \left (\frac{f_1}{f_0},\cdots, \frac{f_m}{f_0}\right )(z)\right |\,\le\, \epsilon ,~~~\forall z\in U_j.$$
We choose $m\ge n$ and $f_1=z_1,\cdots, f_n = z_n$. It follows from the above that $F_{i,j}^m$ is a biholomorphic map
from each $\tilde\phi_i(\vartheta_j(U_j))$ onto its image which contains a ball of radius close to $8$ in the cone $\CC^{n-1}\times {\cal C}_y'$.
We will abbreviate $F_{i,j}^n$ by $F_{i,j}$.

For $\epsilon$ sufficiently small and $i$ sufficiently large, when restricted to $B_{8 s_j r_i}(x_i, \omega_i)$, the map $F_{i,j}^m$ is one-to-one on outside a small tubular neighborhood of ${\cal S}_x$. Then by using the above {\bf 1} and {\bf 2}, one can see that each $F_{i,j}$ is a biholomorphic map from $B_{8 s_j r_i}(x_i, \omega_i)$ onto its image
which contains the following set
$$U_j'\,=\,\{\,(z',z_n)\,\in \, \CC^{n-1}\times {\cal C}_y'\,|\,\,\sqrt{|z'|^2 + |z_n|^{2\bar\beta}} \,< 8-\epsilon\,\}.$$

It follows from the above derivative estimate on $S_{i,j}^b$ that
$$ \sup_{B_{8 s_j r_i}(x_i, \omega_i)} | d F_{i,j}^m |_{  \omega_i}\,\le\, C_m\,(s_j\, r_i)^{-2},$$
where $C_m$ is a constant independent of $i$ and $j$. This is equivalent to
\begin{equation}\label{eq:lb-1}
 \omega_0 \,\le\, C_m \,(s_j\, r_i)^{-2}\,\omega_i,
 \end{equation}
where $\omega_0$ denotes the Euclidean metric on $\CC^m$. A consequence of this is that by taking a subsequence if necessary,
as $i$ goes to $\infty$, we get a limiting map
$$F_{\infty,j}^m: B_{8 s_j}(y, g_x)\,\mapsto\, \CC^m.$$
Moreover, its image is the subvariety $V_j^m\subset \CC^m$ which coincides with the limit of $F_{i,j}^m(B_{8 s_j r_i}(x_i, \omega_i))$.
Such a limit exists because of the well-known Bishop theorem in complex analysis and the following volume bound:
$$\int_{F_{i,j}^m(B_{8 s_j r_i}(x_i, \omega_i))} \,\omega_0^n \,\le\, C_m\,\frac{{\rm vol}(B_{8 s_j r_i}(x_i, \omega_i))} {(s_j\, r_i)^{2n}} \,\le\, C_m'.$$
The last one follows from the volume comparison.

Next we show that for $j$ sufficiently large, $F_{i,j}(D\cap B_{7 s_j r_i}(x_i, \omega_i)) $ converge to a local divisor $D_{j}^n\subset \CC^n $.
Again it is a corollary of the Bishop theorem, for this purpose, it suffices to bound the volume of $F_{i,j}(D\cap B_{7 s_j r_i}(x_i, \omega_i)) $.
Since $({\cal C}_x, s_j^{-2} g_x, y)$ converge to the standard cone $\CC^{n-1}\times {\cal C}_y'$ with the standard metric $g_{\bar\beta}$,
for $j,i$ sufficiently large, the image of $D\cap B_{8 s_j r_i}(x_i, \omega_i)$ under the map $F_{i,j}$ lies
in a tubular neighborhood:
$$T_{8,\epsilon}\,=\,\{ (z', z_n)\,|\, |z'| < 8, |z_n| < \epsilon\}.$$
On the other hand, using the slicing argument as that in \cite{cheegercoldingtian}, one can show that for each fixed $z'$ with $|z'| < 7.5$, the line segment $\{(z',z_n)\,|\,|z_n| \le 6\}$ intersects with $F_{i,j}(D\cap B_{8 s_j r_i}(x_i, \omega_i))$ at $k$ points (counted with multiplicity), where $(1-\bar\beta)= k(1-\beta_\infty)$. 

It is now easy to bound the volume of $F_{i,j}(D\cap B_{7 s_j r_i}(x_i, \omega_i))$: Let $\tilde\eta:\RR \mapsto \RR$ be a cut-off function
such that $\tilde\eta(t) = 1$ for $t \le 7.3$, $\tilde\eta(t) = 0 $ for $t > 7.8$ and $|\tilde\eta'| \le 2$, then the volume
of $F_{i,j}(D\cap B_{7 s_j r_i}(x_i, \omega_i))$ is bounded from above by
\begin{eqnarray}
&&\int_{F_{i,j}(D\cap B_{8 s_j r_i}(x_i, \omega_i)) } \,\tilde\eta(|z'|)\, (\omega_0 + \sqrt{-1} \partial \bar\partial\, |z_n|^{2\bar\beta} )^{n-1}\nonumber\\
&\le&
\int_{F_{i,j}(D\cap B_{8 s_j r_i}(x_i, \omega_i)) } \,(\tilde\eta  +  |z_n|^{2\bar\beta} \tilde\eta')(|z'|)\,\omega_0^{n-1}\,\le\, 3 k\, 8^{2n }.
\end{eqnarray}

One can easily see that $F_{\infty, j}({\cal S}_x\cap B_{7 s_j}(y, g_x)) $ coincides with $D_j^n$.
\footnote{Here we use the fact that the limit of $D$ coincides with ${\cal S}_x$ modulo a
subset of Hausdorff codimension at least $4$ under the Gromov-Hausdorff convergence of $(M,r_i^{-2} \omega_i, x_i)$ to $({\cal C}_x, \omega_x, o)$. Clearly, the limit
lies in ${\cal S}_x$. On the other hand, by \cite{cheegercoldingtian}, there is no singular point of ${\cal C}_x$ outside the limit of $D$ for which there is a tangent cone of
type $\CC^{n-1}\times {\cal C}_y'$.} We can also prove that for any $m > n$,
$F_{i,j}^m(D\cap B_{7 s_j r_i}(x_i, \omega_i)) $ converge to a local divisor $D_{j}^m\subset V_j^m\subset \CC^m $.

For convenience, we summarize the above as follows with one extra property.
\begin{lemm}
\label{lemm:limit2}
For any $\epsilon > 0$ small, there is a $j_\epsilon$ such that for any $j\ge j_\epsilon$, the Lipschtz map $F_{\infty,j}$ maps to $B_{7 s_j}(y, g_x))$
into $B_{7+\epsilon}(o, g_{\bar\beta})$ satisfying: 

\vskip 0.1in
\noindent
(1) Its image contains $B_{7-\epsilon}(o, g_{\bar\beta})$; 

\vskip 0.1in
\noindent
(2) $F_{\infty, j}({\cal S}_x\cap B_{7 s_j}(y, g_x)) $
is a local divisor $D_j^n$ which is contained in a tubular neighborhood $T_{8,\epsilon}$; 

\vskip 0.1in
\noindent
(3) For any $\delta >0$, there is an $\epsilon'=\epsilon'(\delta)$ such that
$F^{-1}_{\infty,j} (T_{6,\epsilon'}) \subset T_{\delta} ({\cal S}_x)\cap B_{(6+\epsilon) s_j} (y,g_x)$.

\end{lemm}
\begin{proof}
I have shown the validity of (1) and (2). For (3), we can prove by contradiction. If not true, then $F^{-1}_{\infty, j}(V_j^n\cap B_{6.5} (o,g_{\bar\beta}))$ has
at least two distinct components, one lies in ${\cal S}_x$ while another is not. This implies that for $i$ sufficiently large,
the pre-image $F_{i,j}^{-1}(F_{i,j}(D)\cap B_{6.5} (o,g_{\bar\beta})$ has at least two components, which contradicts to the fact that $F_{i,j}$ is one-to-one
on $B_{7 s_j r_i}(x_i, \omega_i)$.
\end{proof}

Next we observe: For $i, j$ sufficiently large, there are
uniformly bounded functions $\varphi_{i,j}$ on $B_{8 s_j r_i}(x_i, \omega_i)$ satisfying:
\begin{equation}
\label{eq:potential}
(s_j r_i)^{-2} \omega_i\, = \,\sqrt{-1}\, \partial \bar\partial\, \varphi_{i,j}~~~{\rm on}~~B_{8 s_j r_i}(x_i, \omega_i).
\end{equation}
This is because of the almost constancy of $S_{i,j}^0$. 
A consequence of this observation is that the volume of $D\cap B_{7 s_j r_i}(x_i, \omega_i)$ with respect to $(s_j r_i)^{-2} \omega_i$
is uniformly bounded. In fact, we can prove more.

\begin{lemm}
\label{lemm:est4}
We adopt the notations above. Assume that (1) $\xi:\RR\mapsto [0,1]$ is a smooth function with $\xi(t) = 1$ for any $t \ge 8 \epsilon$ and (2) $f$ is
a holomorphic function on $F_{\infty,j}(B_{7 s_j}(y, g_x))$ such that $|f(z',z_n)| \ge |z_n|$ whenever $|z_n| \ge 8 \epsilon$. Then there is a uniform
constant $C$ such that
$$s_j ^{2-2n}\,\int_{ B_{6 s_j }(y, g_x)} |\nabla ( h\cdot F_{\infty,j})|^2_{\omega_x} \,\omega_x^n\,\le\, C \,\int_{F_{\infty,j}(B_{7 s_j }(y, g_x)) }
\sqrt{-1}\, \partial h\wedge \bar\partial h\wedge \omega_0^{n-1},$$
where $h(z',z_n)  = \xi\cdot |f|^2 (z',z_n)$ and $\omega_0$ denotes the Euclidean metric on $\CC^{n-1}$.
\end{lemm}
\begin{proof}
It suffices to prove the corresponding inequality for each $F_{i,j}$ and then let $i$ go to $\infty$. As above, let
$\tilde\eta:\RR \mapsto \RR$ be a cut-off function
such that $\tilde\eta(t) = 1$ for $t \le 6.3$, $\tilde\eta(t) = 0 $ for $t > 6.8$, $|\tilde\eta'| \le 2$ and $|\tilde\eta''| \le 4$, then we have
$$\sqrt{-1}\, \partial \bar\partial\, \tilde\eta(|z'|) \,\le\, 12 \,\omega_0,$$
moreover, $\tilde \eta(|z'|) |d h |^2$ vanishes near the boundary of $F_{i,j}(B_{7 s_j r_i}(x_i, \omega_i))$. By the definition of $h$, we also have
$$\partial h \wedge \partial\bar\partial h \,=\,0.$$
Using these facts and integration by parts, we can deduce
\begin{eqnarray}
&&(s_j r_i)^{-2n}\,\int_{B_{7 s_j r_i}(x_i, \omega_i) } \,\eta(|z'|) \,|\nabla (h\cdot F_{i,j})|^2_{\omega_i} \,\omega^n_i\nonumber\\
&=&n\,\int_{F_{i,j}(B_{7 s_j r_i}(x_i, \omega_i))} \,\eta(|z'|) \,\sqrt{-1} \,\partial h \wedge \bar \partial h \,\wedge\,(\sqrt{-1}\,
\partial \bar\partial (\varphi_{i,j}\cdot F_{i,j}^{-1}) )^{n-1}\nonumber\\
&\le&
C\,\int_{F_{i,j}(B_{7 s_j r_i}(x_i, \omega_i)) } \,\sqrt{-1} \,\partial h\wedge \bar \partial h \wedge \omega_0^{n-1}.
\end{eqnarray}
Then the lemma follows.

\end{proof}

Now we can complete the proof of Lemma \ref{lemm:cut-off-2}. The arguments are similar to those of the proof for the case with Assumption ${\bf A}_1$.
For the readers' convenience, we repeat some of them here.

For any small $\epsilon_0 >0$, since $\bar {\cal S}_{x}$ has vanishing Hausdorff measure of dimension strictly bigger than $2n-4$,
we can find a finite cover of $\bar{\cal S}_{x} \cap B_{\bar\epsilon^{-1}}(x,g_x)$ by balls $B_{ r_a}(y_a,g_x)$ ($a\,=\,1,\cdots,l$)
satisfying:

\vskip 0.1in
\noindent
(i) $y_a \in \bar {\cal S}_{x}$ and $2 r_a \le \epsilon_0$;

\vskip 0.1in
\noindent
(ii) $B_{r_a/2}(y_a,g_x)$ are mutually disjoint;

\vskip 0.1in
\noindent
(iii) $\sum_a r_a^{2n-3} \,\le\, 1$;

\vskip 0.1in
\noindent
(iv) The number of overlapping balls $B_{2r_a}(y_a,g_x)$ is uniformly bounded.
\vskip 0.1in

We denote by $\bar\eta$ a cut-off function:  $\RR \mapsto \RR$
satisfying: $0\,\le\, \bar \eta \,\le\, 1$, $|\bar \eta'(t)| \,\le\, 2$ and
$$\bar\eta(t)\,=\, 1~~{\rm for}~~t \,>\, 1.6 ~~{\rm and}~~\bar\eta(t) \,=\, 0 ~~{\rm for}~~ t \,\le \,1.1. $$

As before, we set $\chi\,=\,\prod_a \chi_a$, where
$$\chi_a(y)\,=\,\bar\eta \left(\frac{d(y , y_a)}{r_a}\right)~~{\rm if}~y\in B_{2 r_a}(y_a,g_x)~~~{\rm and}~~~\chi_a(y)\,=\,1~~{\rm otherwise}.$$
Then $\chi$ vanishes on the closure of $B\,=\,\cup_{a} B_{r_a}(y_a,g_x)$ which contains $\bar {\cal S}_{x}\cap B_{\bar\epsilon^{-1}}(x,g_x)$, furthermore,
$\chi$ satisfies
\begin{equation}
\label{eq:prod-a1}
\int_{{\cal C}_x} |\nabla \chi|^2 \, \omega_x^n \,\le \, C \,\epsilon_0,
\end{equation}
where $C$ is a uniform constant.

There is a finite cover of ${\cal S}_x\cap B_{\bar\epsilon^{-1}}(x,g_x) \backslash B$ by balls $B_{6 s_b}(y_b, g_x)$ for which Lemma \ref{lemm:limit2} holds
($b=1,\cdots, N$).
We may assume that the number of overlapping balls $B_{6 s_b}(y_b,g_x)$ is bounded.
Choose smooth functions $\{\zeta_b\}$ associated to the cover $\{B_{6 s_b}(y_b, g_x)\}$ satisfying: 

\vskip 0.1in
\noindent
(1) $0\le \zeta_b \le 1$; 

\vskip 0.1in
\noindent
(2) ${\rm supp}(\zeta_b)$ is contained in $B_{6 s_b}(y_b, g_x)$; 

\vskip 0.1in
\noindent
(3) $\sum_b \zeta_b \equiv 1$ near ${\cal S}_x\cap B_{\bar\epsilon^{-1}}(x,g_x) \backslash B$.
\vskip 0.1in
\noindent
Therefore, $\{\zeta_b\}, 1 -\sum_b \zeta_b$ form a partition of unit for the cover $\{B_{6 s_b}(y_b, g_x)\}$ and $B_{\bar\epsilon^{-1}}(x,g_x) $.

As before, we denote by $\eta$ a cut-off function: $\RR \mapsto \RR$ satisfying:
$0\,\le\, \eta \,\le\, 1$, $|\eta'(t)| \,\le\, 1$ and
$$\eta(t)\,=\, 0~~{\rm for}~~t \,>\, \log(-\log \delta^3) ~~{\rm and}~~\eta(t) \,=\, 1 ~~{\rm for}~~ t \,< \,\log (-\log \delta) . $$
For each $b$, by Lemma \ref{lemm:limit2}, there is a divisor $D_b^n\subset B_6(0,g_{\bar\beta_b})$,
where $(1-\bar\beta_b)=k_b (1-\beta_\infty)$. Choose a local defining function $f_b$ of $D_b^n$ satisfying (2) in Lemma \ref{lemm:est4}.
We define a function $\gamma_{\bar\epsilon,b}$ on $B_6(o,g_x)$ as follows: If $|f_b|(y) \,\ge\, \bar\epsilon /3$, put $\gamma_{\bar\epsilon,b} (y)\,=\, 1$ and if $|f_b|(y)\,<\,\bar\epsilon$, put
\begin{equation}
\label{eq:gamma-a1b}
\gamma_{\bar\epsilon,b}(y)\,=\, \eta \left ( \log\left (-\log \left(\frac{|f_b|(y)}{\bar\epsilon} \right)\right)\right).
\end{equation}
Then we put
\begin{equation}
\label{eq:gamma-a1}
\gamma_{\bar\epsilon}(y)\,=\, \chi(y)\,(1- \sum_b \zeta_b(y) \,+ \,\sum _b \zeta_b (y) \,\gamma_{\bar\epsilon,b}(y) ).
\end{equation}
Clearly, $\gamma_{\bar\epsilon}$ is smooth. If we choose $\epsilon_0$ and $\delta$ sufficiently small, we have
$\gamma_{\bar\epsilon}(y)\,=\,1$ for any $y$ with $d(y, {\cal S}_x) \,\ge\, \bar\epsilon$, also $\gamma_{\bar\epsilon}$ vanishes in a neighborhood of ${\cal S}_x$.
Furthermore, by using \eqref{eq:prod-a1}, Lemma \ref{lemm:limit2} and Lemma \ref{lemm:est4}, we can also show
$$\int_{B_{{\bar\epsilon}^{-1}}(o,g_x)} \,|\nabla \gamma_{\bar\epsilon}|^2\, \omega_x^n\,\le\, \bar\epsilon.$$
Thus, the proof of Lemma \ref{lemm:cut-off-2} is completed.

\vskip 0.1in
There are other ways of completing the proof of Lemma \ref{lemm:cut-off-2}. One is to verify Assumption ${\bf A}_1$ (cf. Remark \ref{rema:c-x}.
Another is to estimate the volume of tubular neighborhood of ${\cal S}^0_x$. 
Let us outline it in the following. 

For any small $\epsilon_0 >0$, we can find a finite cover of $\bar {\cal S}_{x} \cap B_{\bar\epsilon^{-1}}(x,g_x)$ by balls $B_{ r_a}(y_a,g_x)$ ($a\,=\,1,\cdots,l$) with properties (i)-(iv) as above. Then we can have a smooth function $\chi$ associated to this covering as we did above.
Put $\rho(y)\,=\,d(y, {\cal S}_{x})$ and
$$K \,=\, \overline{B_{{\bar\epsilon}^{-1}}(o,g_x)}\backslash \cup_a B_{ r_a/2}(y_a,g_x).$$
Define $\gamma_{\bar\epsilon} $ according to \eqref{eq:gamma-1}. Clearly, it satisfies (1) and (2) in Lemma \ref{lemm:cut-off-2}. For (3), if $\delta$ is sufficiently small, we
only need to prove
\begin{equation}
\label{eq:gamma-2}
\int_{K} |\nabla \eta\cdot \zeta|^2 \,\omega^n_x\,=\,\int_K |\eta'\cdot \zeta|^2 \,|\nabla \zeta|^2\,\omega_x^n \,\le\, \frac{\bar\epsilon}{2}\,,
\end{equation}
where
$$\zeta(y)\,=\,\log\left (-\log \left(\frac{\rho(y)}{\bar\epsilon} \right)\right).$$
By the well-known co-area formula, we have
$$\int_{K} |\eta'\cdot \zeta|^2\,|\nabla \zeta|^2 \,\omega^n_x \,=\,\int_0^\infty |\eta'(r)|^2\,|\nabla \zeta| \,{\rm Vol}(\zeta^{-1}(r)\cap K) \,dr .$$
Clearly, $\zeta(y)\,=\,r$ implies that $\rho(y) \,=\, \bar\epsilon \,e^{- e^r}$, moreover, if we set $s = |\nabla \zeta|^2$, we have
$$\sqrt{s} \,=\, \frac {1}{ \bar\epsilon}\, e^{e^r- r}.$$
It is a monotonic function for $r > 0$, thus, the inverse $r = r(s)$ exists. By the co-area formula again, we have
$$\int_0^\infty ds \int_{\{ |\nabla \zeta|^2 \ge s\}\cap K } |\eta'\cdot\zeta|^2 \,\omega_x^n \,=\, \int_0^\infty ds \int_{r(s)}^\infty \frac{|\eta'(r)|^2}{|\nabla\zeta| (r)} \,{\rm Vol}(\zeta^{-1}(r)\cap K) \, dr.$$
Exchanging the order of integrals on $r$ and $s$, we get
$$\int_{K} |\eta'\cdot \zeta|^2\,|\nabla \zeta|^2 \,\omega^n_x \,=\, \int_0^\infty ds \int_{\{ |\nabla \zeta|^2 \ge s\}\cap K} |\eta'\cdot\zeta |^2 \,\omega_x^n.$$
Combining the above integrals, we get
$$\int_K |\nabla (\eta\cdot\zeta)|^2\,\omega_x^n\,\,=\,\int_0^\infty s'(t)\, dt \,\int_{\{|\nabla \zeta |^2 \ge s(t)\}\cap K } |\eta'\cdot\zeta|^2 \,\omega_x^n,$$
where
$$s (t) = \frac{1}{ \bar\epsilon ^2 t^2 \,(-\log t)^2} .$$
Since $\eta'(\zeta(y)) = 0$ unless $\bar\epsilon \,\delta^3 \,\le\,\rho(y)\,\le\, \bar\epsilon \,\delta$, we can deduce from this identity and
the following lemma that
$$\int_K |\nabla (\eta\cdot\zeta)|^2\,\omega_x^n\,\le\, C_K \left (\int_0^\delta \bar\epsilon^2 t^2\, s'(t)\, dt \,+\,\int_\delta^\infty \bar\epsilon^2 \delta^2\, s'(t)\, dt\right).$$
We get \eqref{eq:gamma-2} from this estimate since the last two integrals tend to $0$ as $\delta$ goes to $0$. Therefore, Lemma \ref{lemm:cut-off-2} follows from the following.

\begin{lemm}
\label{lemm:a2-a3}
For any compact subset $K\subset {\cal C}_x\backslash \bar {\cal S}_x$, there is a constant $C_K$ such that for any $r < 1$, the volume
of $T_r({\cal S}_x)\cap K$ is bounded by $C_K r^2$, where $T_r({\cal S}_x) \,=\,\{ z\,|\, d(z, {\cal S}_x) \le r\}$.
\end{lemm}
This follows from an estimate on the lower bound of the ratio $r^{2-2n} {\rm vol}({\cal S}_x\cap B_r(y,g_x))$ for any $r\le 1$ and $y \in K\cap {\cal S}_x$.
Such an estimate can be easily derived by a blow-up argument and what we have obtained above.

\begin{rema}
\label{rema:c-x}
In fact, Assumption ${\bf A}_1$ can be established by the techniques used above. An approach is to use the maps $F_{\infty,j}^m$. 
One can show that each composition $F_{\infty,j}^n\cdot (F_{\infty,j}^m)^{-1}: V^m_j\,\mapsto V_j^n$
is well-defined and the limit of
$F_{i,j}^n\cdot (F_{i,j}^m)^{-1}$. Each such map supposes to be finite and one-to-one on a sufficiently large open subset.
One can deduce from these that $F_{\infty,j}^n\cdot (F_{\infty,j}^m)^{-1}$ is one-to-one. It implies that $F_{\infty,j}$ is an one-to-one map. Then we 
get what we wanted.
\end{rema}

\end{document}